\documentclass[12pt,reqno]{amsart}
\usepackage{amssymb}
\usepackage{graphicx}
\usepackage{amsmath}
\usepackage{a4wide}
\usepackage{version}
\usepackage{mathrsfs}
\usepackage{tikz}
\usepackage{tikz-cd}
\usepackage{mathtools}
\usetikzlibrary{arrows,decorations.markings, matrix}
\usepackage{float}

\usepackage[bookmarks=true, bookmarksopen=true,%
bookmarksdepth=3,bookmarksopenlevel=2,%
colorlinks=true,%
linkcolor=blue,%
citecolor=blue,%
filecolor=blue,%
menucolor=blue,%
urlcolor=blue]{hyperref}
\usetikzlibrary{decorations.pathreplacing}

\usepackage[all]{xy}

\numberwithin{equation}{subsection}

\theoremstyle{definition}

\newtheorem*{theorem*}{Theorem}
\newtheorem*{lemma*}{Lemma}
\newtheorem*{prop*}{Proposition}

\newcommand{\bZ}{\mathbf{Z}}
\newcommand{\bR}{\mathbf{R}}
\newcommand{\bP}{\mathbf{P}}

\newcommand{\Hom}{\operatorname{Hom}}

\newcommand{\Ext}{\operatorname{Ext}}

\newcommand{\cO}{\mathcal{O}}

\newcommand{\coker}{\operatorname{coker}}

\newcommand{\Coh}{\operatorname{Coh}}

\newcommand{\Spec}{\operatorname{Spec}}

\newcommand{\Proj}{\operatorname{Proj}}

\newcommand{\Pic}{\operatorname{Pic}}

\renewcommand{\ker}{\operatorname{ker}}

\newcommand{\bF}{\mathbf{F}}

\newcommand{\CF}{\mathrm{CF}}
\newcommand{\HF}{\mathrm{HF}}

\newcommand{\cM}{\mathcal{M}}
\newcommand{\bC}{\mathbf{C}}
\newcommand{\bQ}{\mathbf{Q}}
\newcommand{\cE}{\mathcal{E}}

\newcommand{\cW}{\mathcal{W}}

\newcommand{\Ls}[1]{L_{(#1)}}
\newcommand{\area}{\mathrm{area}}
\newcommand{\mas}{\mathrm{mas}}
\newcommand{\Fuk}{\mathrm{Fuk}}
\newcommand{\FF}{{\mathit{FF}}}

\newcommand{\tabs}{\theta^{\mathrm{abs}}}

\title{A symplectic look at the Fargues-Fontaine curve}
\author{Yank\i\ Lekili and David Treumann}

\begin{document}

\maketitle

\section{Introduction}

This paper discusses homological mirror symmetry for the Fargues-Fontaine curve.

\subsection{The Fargues-Fontaine curve}
\label{intro:FF}
Let $E$ be a local field and let $C$ be a perfectoid field of characteristic $p$.  For each such pair $(E,C)$, Fargues and Fontaine have defined an $E$-scheme that we will denote by $\FF_E(C)$ --- it is denoted by $X_{C,E}$ in \cite[Def. 6.5.1]{FF}.  It is in no sense a ``curve over $E$'' or even a variety: it is not of finite type over $E$ or over any other field.  A scheme that is not of finite type over some base cannot be smooth or proper in the usual sense (\cite[Vol 4, \S 17; Vol 2, \S 5]{EGA}) and yet $\FF_E(C)$ resembles a closed Riemann surface is some peculiar ways:
\begin{itemize}
\item It is noetherian of Krull dimension one.  Moreover it is regular, so that the local ring at each closed point of $\FF_E(C)$ has a discrete valuation.
\item A nonzero rational function $f$ (that is, a section of $\cO_{\FF}$ over the generic point) has $v(f) \neq 0$ for at most finitely many of these valuations $v$, and $\sum_v v(f) = 0$
\end{itemize}
In fact $\FF_E(C)$ even resembles the Riemann sphere: one has
\[
\Pic(\FF_E(C)) = \bZ \text{ and } H^1(\cO_{\FF}) = 0
\]
There are some contrasts with the Riemann sphere: $\FF_E(C)$ has indecomposable vector bundles of higher rank, and its \'etale fundamental group is naturally isomorphic to the absolute Galois group of $E$.  Fargues has a program to apply these properties of $\FF_E(C)$ to the local Langlands correspondence \cite{Fargues}.

When $E = \bQ_p$, $\FF_E(C)$ is an important object in $p$-adic cohomology --- it was introduced to organize some of the structures of $p$-adic Hodge theory.  When $E = \bF_p(\!(z)\!)$, the analogous structures are those of Hartl \cite{Hartl}.   We have nothing to say about $\FF_{\bQ_p}(C)$ but we are able to touch $\FF_{\bF_p(\!(z)\!)}(C)$ with mirror symmetry.

\subsection{Homological mirror symmetry}
Homological mirror symmetry (HMS) is a framework for relating Lagrangian Floer theory on a symplectic manifold to the homological algebra of coherent sheaves on a scheme --- often, a scheme that is seemingly unrelated to the symplectic manifold.

What symplectic structure could be mirror to $\FF_{\bF_p(\!(z)\!)}(C)$?  We suggest the answer is a two-dimensional torus.  There is already a very well-studied mirror relationship between the symplectic torus and the Tate elliptic curve (over $\bZ(\!(t)\!)$), which we review in \S\ref{two}.  To get the Fargues-Fontaine curve in place of the Tate curve, we introduce two changes:
\begin{enumerate}
\item We couple Lagrangian Floer theory to a locally constant sheaf of rings on the torus --- the fiber of this sheaf of rings has characteristic $p$, and going around one of the circles is the $p$th power map.  (Going around the other circle is the identity map).
\item We set the Novikov parameter (this is the element $t \in \bZ(\!(t)\!)$ in the ground ring of the Tate curve) to $t = 1$ --- symplectically this is sort of like studying the limit as the symplectic form goes to $0$.
\end{enumerate}
Both of these maneuvers are unusual in symplectic geometry.  The first turns out to be straightforward, so that one obtains a Fukaya $A_{\infty}$-category with the usual properties.  The second is much more delicate and touches some folklore questions about ``convergent power series Floer homology.''

\subsection{Lagrangian Floer theory on the torus}
\label{intro:13}
Let $T$ be a $2$-dimensional torus, which we present as a quotient of $\bR^2$ by $\bZ^2$ and endow with the standard symplectic form $dx\,dy$. For each integer $m$, let $\Ls{m} \subset T$ denote the image of the line in $\bR^2$ through the origin of slope $-m$.  Let $\Ls{\infty}$ denote the image of the vertical line through the origin.  We orient $\Ls{m}$ from left to right and $\Ls{\infty}$ from top to bottom.  The figure shows $\Ls{0}$, $\Ls{\infty}$, and $\Ls{3}$ in a fundamental domain of $T$:
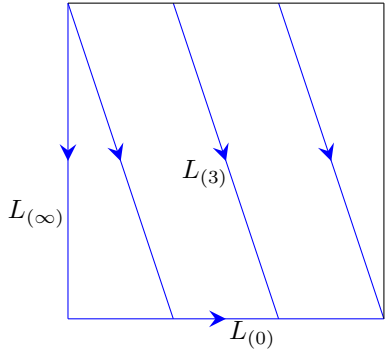
 \begin{center}
\begin{tikzpicture}[scale=0.7]
      \tikzset{->-/.style={decoration={ markings,
        mark=at position #1 with {\arrow[scale=2,>=stealth]{>}}},postaction={decorate}}}
         \draw[blue, ->-=.5]  (0,0) -- (6,0);
         \draw [blue, ->-=.5] (0,6) -- (0,0);
	\draw [blue, ->-=.5] (0,6)--(2,0);
	\draw [blue, ->-=.5] (2,6)--(4,0);
	\draw [blue, ->-=.5] (4,6)--(6,0);
         \draw (0,6) -- (6,6);
         \draw (6,0) -- (6,6);

\node at (-0.6,2)   {\footnotesize $\Ls{\infty}$};
\node at (3.5,-0.3)   {\footnotesize $\Ls{0}$};
\node at (2.6,2.75) {\footnotesize $\Ls{3}$};
\end{tikzpicture}
\end{center}
If $m_1 > m_0$, then $\Ls{m_1}$ and $\Ls{m_0}$ meet transversely in $(m_1 - m_0)$ points.  Lagrangian Floer theory gives algebraic structures to the free modules
\begin{equation}
\label{eq:000}
\CF(\Ls{m_0}, \Ls{m_1}) := \bigoplus_{x \in \Ls{m_0} \cap \Ls{m_1}} \Lambda,
\end{equation}
where $\Lambda$ is a suitable ring, about which more in \S\ref{intro:nov}.
``$\CF$'' stands for ``Floer cochains.''  The orientations of $\Ls{m_0}, \Ls{m_1}$ endow \eqref{eq:000} with a $\bZ/2$-grading (which can be lifted to a $\bZ$-grading by making some additional topological choices), and $\CF^*(\Ls{m_0},\Ls{m_1})$ supports a differential of degree $+1$ \S\ref{subsec:tsfccam}.  In the case at hand these gradings are concentrated in a single degree 0, and the differential is zero.

We will return to the differential in \S\ref{intro:degzeropiece} (and a little in \S\ref{intro:HFL0L0}), but to start we will be very interested in the ``triangle product'':
\begin{equation}
\label{eq:111}
\CF(\Ls{m_1},\Ls{m_2}) \times \CF(\Ls{m_0},\Ls{m_1}) \to \CF(\Ls{m_0},\Ls{m_2})
\end{equation}
whose value on $(x_1,x_2) \in (\Ls{m_1} \cap \Ls{m_2}) \times (\Ls{m_0} \cap \Ls{m_1})$ is the summation
\begin{equation}
\label{eq:112}
\sum_{y \in \Ls{m_0} \cap \Ls{m_2}} y \left(\sum_{u \in \cM(y,x_2,x_1)} \pm t^{\mathrm{area}(u)} \right)
\end{equation}
The inner sum is infinite: it is indexed by the set of rigid pseudoholomorphic triangles
\begin{equation}
u:\Delta^2 \to T
\end{equation}
with vertices at $x_1,x_2,y$ and one edge each along $\Ls{m_2},\Ls{m_1},\Ls{m_0}$.  The sign in $\pm t^{\area(u)}$ depends on $u$, and on the choice of a spin structure on each of the oriented $1$-manifolds $\Ls{m_i}$, see \S\ref{sorp}.  In the present case it is possible to make those choices so that the signs are identically $+1$.

\subsection{Dehn twist}
\label{intro:dehn}
There is a canonical identification of $\CF(\Ls{m},\Ls{n})$ with $\CF(\Ls{0},\Ls{n-m})$, induced by
\[
(x,y) \mapsto (x,y-mx)
\]
the $m$-fold Dehn twist around $\Ls{\infty}$.  An old suggestion of Seidel's \cite{Zaslow} is to use this identification to package the triangle products as a graded ring structure on the sum
\begin{equation}
\label{eq:121}
\Lambda \oplus \bigoplus_{m = 1}^{\infty} \CF(\Ls{0},\Ls{m})
\end{equation}
The multiplication on \eqref{eq:121} is associative and commutative for nontrivial reasons.  The associativity is a consequence of a very general Floer-theoretic argument that studies $1$-dimensional moduli spaces of pseudoholomorphic quadrilaterals \S\ref{ainfrelations}, and the commutativity is a consequence of a more particular observation about the Dehn twist \cite[\S 3]{Zaslow}.  

\subsection{The Floer cohomology of $(\Ls{0},\Ls{0})$}  
\label{intro:HFL0L0}
In \eqref{eq:121}, we have inserted the unit of the ring by hand (the summand $\Lambda$, which we place in degree zero), but this can also be motivated Floer-theoretically.  The definition of $\CF$ in \eqref{eq:000} is not the right one when $\Ls{m_0} = \Ls{m_1}$, or for any other pair that do not meet transversely.  But if $\phi = \{\phi^s\}_{s \in \bR}$ is a general Hamiltonian isotopy, then
\[
\CF(\phi^s \Ls{0},\Ls{0}) := \bigoplus_{x \in \phi^s \Ls{0} \cap \Ls{0}} x \cdot \Lambda
\]
together with its differential, gives a cochain complex whose cohomology groups do not depend on $\phi$.  These cohomology groups are $\bZ/2$-graded, the $\Lambda$ piece of \eqref{eq:121} is naturally identified with $\HF^0(\Ls{0},\Ls{0})$, cf. \S\ref{ex-idmap}.

\subsection{Novikov ring $\Lambda$ and Floer theory relative to a divisor}
\label{intro:nov}
There is some flexibility in choosing the ground ring $\Lambda$, but it should contain a ring of constants (let us use $C$ for this ring --- later on it will be the same as the $C$ of \S\ref{intro:FF}) a parameter $t$ and all necessary powers of it, and it should carry a topology in which all the sums \eqref{eq:121} converge.  The conventional choice is the Novikov ring \eqref{eq:LamZ0}, which we will denote by $\Lambda_C$:
\begin{equation}
\label{eq:LamC}
\Lambda_C =\left\{\sum_{i = 0}^{\infty} a_i t^{\lambda_i} \mid a_i \in C, \lambda_i \in \bR \text{ and } \lim_{i \to \infty} \lambda_i = \infty\right\}
\end{equation}

We can shrink those coefficients to $C[\![t]\!]$ by the following device of Seidel's, called Floer theory ``relative to a divisor.'' Rather than computing the area of the triangles $u$, we fix a basepoint $D \in T^2$ (in general, a symplectic divisor $D \subset T^2$) and use the cardinality of $u^{-1}(D)$ in place of symplectic area.  If $D$ is in the first quadrant and extremely close to $(0,0)$, then this cardinality is given by a simple formula which is independent of $D$ unless the triangle $u$ is extremely acute --- let us write $\area_{\bZ}(u)$ for this discretized notion of area.  With some additional care, by letting $D \to (0,0)$ (see \cite[\S 7.2.3]{LP} and \cite[Prop. 9.1]{LP}), we get a graded ring
\begin{equation}
\label{eq:131}
C[\![t]\!] \cdot 1 \oplus \bigoplus_{m = 1}^{\infty} \CF(\Ls{0},\Ls{m})
\end{equation}

\subsection{Theorem \cite{LP}}
\label{intro:LP-theorem}
The $C[\![t]\!]$-scheme
\begin{equation}
\Proj\left(C[\![t]\!] \cdot 1 \oplus \bigoplus_{m = 1}^{\infty} \CF(\Ls{0}, \Ls{m})\right)
\end{equation}
is isomorphic to $E_{\mathrm{Tate}} \times_{\bZ[\![t]\!]} C[\![t]\!]$; the Tate elliptic curve over $C[\![t]\!]$ whose Weierstrass equation is
\[
y^2 + xy = x^3 - b_2 x - b_3
\]
where $b_2,b_3 \in \bZ[\![t]\!]$ are the series
\begin{equation}
\label{eq:142}
b_2 = \sum_{n = 1}^{\infty} 5n^3 \frac{t^n}{1-t^n} \qquad b_3 = \sum_{n =1}^{\infty} \left(\frac{7n^5 + 5n^3}{12}\right) \frac{t^n}{1-t^n}
\end{equation}

\subsection{Theta series}
\label{intro:theta}
The relationship between \eqref{eq:131} and functions on the Tate elliptic curve is more transparent when those functions are described in terms of $\theta$-series.  Set
\begin{equation}
\label{eq:ttabs}
\theta_{m,k} := \sum_{i = -\infty}^{\infty} t^{m \frac{i (i-1)}{2}+ k i}\, z^{mi + k}; \qquad \theta_{m,k}^{\mathrm{abs}} := \sum_{i = -\infty}^{\infty} t^{(mi+k)^2/(2m)} z^{mi+k}
\end{equation}
The simplest of these series is the Jacobi function
\[
\theta_{1,0} = \sum_{i = -\infty}^{\infty} t^{\frac{i(i-1)}{2}} z^i = (1+z) \prod_{i = 1}^{\infty} \left[(1-t^i)(1+t^i z)(1+t^i z^{-1})\right]
\]
The others $\theta_{m,k}$ are obtained by a change of variables from $\theta_{1,0}$.  
These series are doubly infinite in $z$, but in formally expanding the product of two of them, the coefficient of $z^i t^j$ has only finitely many nonzero contributions.  The $C[\![t]\!]$-linear span of the $\theta_{m,k}$ (resp. the $\Lambda_C$-linear span the $\tabs_{m,k}$) is closed under multiplication and graded by $m$, and is isomorphic as a graded ring to \eqref{eq:131} (or \eqref{eq:121} in the absolute case).  The isomorphisms send $(k/m,0) \in \Ls{0} \cap \Ls{m}$ to $\theta_{m,k}$ or to $\tabs_{m,k}$.

\subsection{Fukaya category and homological mirror symmetry}
\label{intro:fuk}
The triangle product \eqref{eq:111} resembles a composition law in a category.  It is part of a sequence of structures on the $\CF(L,L')$,
\begin{equation}
\mu_n:\CF(L_{n-1},L_n) \times \cdots \times \CF(L_0,L_1) \to \CF(L_0,L_n)
\end{equation}
that are obtained by summing over the $(n+1)$-gons with sides along $L_0,L_1,\ldots,L_n$ \S\ref{two:polygon}.  When one takes extra care to treat sets of Lagrangians that are not transverse \S\ref{ex-idmap}, these structures define an $A_{\infty}$-category.  After passing to a triangulated envelope and splitting idempotents, we will call any of these $A_{\infty}$-structures a ``Fukaya category'' and denote it by $\Fuk(T)$ (in the absolute case) or $\Fuk(T,D)$ (in the relative case).

Kontsevich's homological mirror symmetry conjecture, specialized to $T$, asks for a quasi-equivalence between $\Fuk(T)$ and the derived category of coherent sheaves on an elliptic curve.  A version of this for complex elliptic curves was obtained in \cite{PZ}.  When $C = \bZ$, Theorem \ref{intro:LP-theorem}, together with a generation result for $\Fuk(T,D)$ \cite[\S 6.3]{LP}, constitute ``homological mirror symmetry over $\bZ$'':
\[
\Fuk(T,D) \cong D^b(\Coh(E_{\mathrm{Tate}}))
\]
The structure sheaf of $E_{\mathrm{Tate}}$ is the image of $\Ls{0}$ under this equivalence.

\subsection{$F$-fields}
\label{intro:ffield}

Let $\underline{\Lambda}$ be a local system of rings on $T$, so that at each point $x \in T$ we are given a ring $\underline{\Lambda}_x$, and along each path $\gamma$ from $x$ to $y$ we are given a ring isomorphism
\begin{equation}
\label{eq:171}
\nabla \gamma:\underline{\Lambda}_x \stackrel{\sim}{\to} \underline{\Lambda}_y
\end{equation}
Suppose that each ring $\underline{\Lambda}_x$ has the structures that we asked for in \S\ref{intro:nov}: it contains a ring of constants (we can denote it by $\underline{C}_x$), distinguished elements of the form $t^a$, and carries a topology.  The maps \eqref{eq:171} should be continuous, carry each $\underline{C}_x$ to $\underline{C}_y$, but leave the elements of the form $t^a$ alone ($\nabla \gamma(t^a) = t^a)$.  

We will develop a version of Floer theory ``with coefficients in $\underline{\Lambda}$.'' 
As in \S\ref{intro:nov}, we could work either relative to a divisor, or absolutely.  In the relative case we would take $\underline{\Lambda} := \underline{C}[\![t]\!]$, where $\underline{C}$ is a locally constant sheaf of rings.  In the absolute case, we would take $\underline{\Lambda} := \Lambda_{\underline{C}}$ \eqref{eq:LamC}.
\medskip

In the example of interest to us, the sheaf of rings is pulled back from $S^1$, along the projection map
\begin{equation}
\label{eq:intro:ffield}
\mathfrak{f}:T \to S^1
\end{equation}
Then $\underline{\Lambda}$ is determined by a ring $C$ and an automorphism (the monodromy around the base $S^1$) $\sigma$ of $C$.  It induces an automorphism of $C[\![t]\!]$ and of $\Lambda_C$ that fixes each $t^a$.  We are interested in the case when $C$ is perfect of characteristic $p$ and $\sigma$ is the $p$th root map.
\medskip

The map \eqref{eq:intro:ffield} (and the monodromy map $\sigma$) is just for book-keeping, but it also has an ``occult'' interpretation, in a way fitting in to the old analogy between number fields and three-manifolds, and between primes and knots.  The generator in the fundamental group of the base circle $S^1$ and the Frobenius in the absolute Galois group of $\bF_p$ act on $C$ in the same way: by $p$th powers.  See \cite{T} for a little bit more about this.  There is also a natural map from the set of closed points of $\FF_E(C)$ (for $E = \bF_p(\!(z)\!)$ or any other local field) to $S^1$, and in some sense this paper explores the idea that the SYZ mechanism for mirror symmetry could apply \S\ref{subsec:syz}.

\subsection{Lagrangian Floer theory --- coupled to $\underline{\Lambda}$}

Let us put
\begin{equation}
\label{eq:intro:172}
\CF(\Ls{m_0},\Ls{m_1};\underline{\Lambda}) := \bigoplus_{x \in \Ls{m_0} \cap \Ls{m_1}} \underline{\Lambda}_x
\end{equation}
Lagrangian Floer theory coupled to $\underline{\Lambda}$ concerns algebraic structures on \eqref{eq:intro:172}, for instance a triangle product
\begin{equation}
\label{eq:173}
\CF(\Ls{m_1},\Ls{m_2};\underline{\Lambda}) \times \CF(\Ls{m_0},\Ls{m_1};\underline{\Lambda}) \to \CF(\Ls{m_0},\Ls{m_2};\underline{\Lambda})
\end{equation}
In some sense \eqref{eq:intro:172} is another free $\Lambda$-module on the intersection points $\Ls{m_0} \cap \Ls{m_1}$, but with many different $\Lambda$-module structures.  The product \eqref{eq:173} is not $\Lambda$-bilinear with respect to any of them.  To define it we give its value on a pair
\begin{equation}
\label{eq:174}
(x_2 \cdot b,x_1 \cdot a) \in \CF(\Ls{m_1},\Ls{m_2};\underline{\Lambda}) \times \CF(\Ls{m_0},\Ls{m_1};\underline{\Lambda})
\end{equation}
for any $a \in \underline{\Lambda}_{x_1}$ and $b \in \underline{\Lambda}_{x_2}$, 
and extend bi-additively (or more precisely, $\Lambda_\bZ$-bilinearly).  The value on $(x_2 \cdot b,x_1 \cdot a)$ is
\begin{equation}
\label{eq:175}
\sum_{y \in \Ls{m_0} \cap \Ls{m_2}} y \sum_{u \in \cM(y,x_2,x_1)} \pm t^{\area(u) \text{ or } \area_{\bZ}(u)} \nabla \gamma_2(b \nabla \gamma_1(a \nabla \gamma_0(1)))
\end{equation}
where $\gamma_0:y \to x_1$, $\gamma_1:x_1 \to x_2$, and $\gamma_2:x_2 \to y$ are the three sides of the triangle $u$, appearing in counterclockwise order.
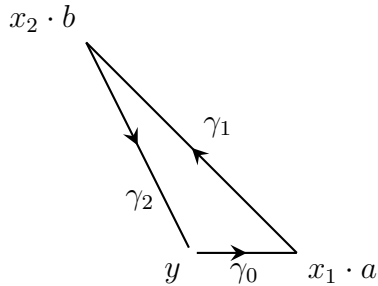
\begin{center}
\begin{tikzpicture}[scale =.7]
\tikzset{->-/.style={decoration={ markings,
        mark=at position #1 with {\arrow[scale=1.5,>=stealth]{>}}},postaction={decorate}}}
\node[below] at (3,0) {$\gamma_0$};
\node[above] at (2.5,2) {$\gamma_1$};
\node[left] at (1.5,1) {$\gamma_2$};
\node[above left] at (0,4) {$x_2 \cdot b$};
\node[below right] at (4,0) {$x_1 \cdot a$};
\node[below left] at (2,0) {$y$};
\draw[black, thick, ->-=.5] (2.1,0)--(4,0);
\draw[black, thick,->-=.5] (0,4)--(2-.05,0+.1);
\draw[black, thick,->-=.5] (4,0)--(0,4);

\end{tikzpicture}
\end{center}
The $\pm$ signs in the formula \eqref{eq:175} are the same as they are in \eqref{eq:112}; in particular one can arrange that they are identically $+1$.
 
When the monodromy of $\underline{\Lambda}$ around $\Ls{\infty}$ is trivial --- equivalently, when $\underline{\Lambda}$ is pulled back along \eqref{eq:intro:ffield} --- it is possible to package these triangle products into a graded ring structure 
\begin{equation}
\label{eq:intro:381}
\bigoplus_{m = 1}^{\infty} \CF(\Ls{0},\Ls{m};\underline{\Lambda})
\end{equation}

\subsection{Theorem}
\label{intro:thm1}
For each $a \in C$, and each pair of integers $m,k$ with $m > k \geq 0$, let $\theta_{m,k}[a]$ denote the formal series
\[
\theta_{m,k}[a] := \sum_{i = -\infty}^{\infty} t^{m\frac{i(i-1)}{2}+ki} z^{mi+k} \sigma^i(a)
\]
Let $\tabs_{m,k}[a]$ denote the formal series
\[
\tabs_{m,k}[a] := \sum_{i = -\infty}^{\infty} t^{\frac{1}{2m} (m i +k)^2} z^{mi+k} \sigma^i(a)
\]
Then the relative (resp. absolute) version of \eqref{eq:intro:381} is isomorphic, as a ring-without-unit, to the $\bZ[\![t]\!]$-linear span of the $\theta_{m,k}[a]$ (resp. to the $\Lambda_{\bZ}$-linear span of the $\tabs_{m,k}[a]$).

\subsection{Specializations of $t$}
Fix a commutative ring $C$ and an automorphism $\sigma$, cf. \eqref{eq:intro:ffield}.  The groups \eqref{eq:intro:172} are linear over $\Lambda_C^{\sigma}$ in the absolute case, and over $C^{\sigma}[\![t]\!]$ in the relative case.  We will discuss the specializations $t = 0$ and $t = 1$.  The case $t = 0$ we treat only briefly in \S\ref{subsec:t=0} --- the absolute case is not of interest, while the relative case is parallel to the ``large volume limit'' of $T$, and its mirror relationship with the nodal cubic curve at the ``large complex structure limit.''

The case $t = 1$ is more delicate.  There is a class of symplectic manifolds and Lagrangian submanifolds (for instance, monotone Lagrangians in a Fano manifold, or in a genus two surface) for which setting $t = 1$ is unproblematic, but the torus does not belong to this class.  And indeed the series \eqref{eq:112}, \eqref{eq:142} do not converge, in any archimedean or nonarchimedean ring, when $t = 1$.

An $F$-field can repair some (but only some) of the convergence.  Define
\begin{equation}
\label{eq:CFCu}
\CF(\Ls{m_0},\Ls{m_1};\underline{C}) = \bigoplus_{x \in \Ls{m_0} \cap \Ls{m_1}} \underline{C}_x
\end{equation}
Setting $t = 1$ in \eqref{eq:175} suggests, in a formal way, a map
\[
\CF(\Ls{m_1},\Ls{m_2};\underline{C}) \times \CF(\Ls{m_0},\Ls{m_1};\underline{C}) \dashrightarrow \CF(\Ls{m_0},\Ls{m_2};\underline{C}) 
\]
When $C$ is complete with respect to a norm $|\cdot|$, $\sigma$ is the $p$th root map, and $m_0 < m_1 < m_2$, this map has a nontrivial domain of convergence.  In particular it defines a multiplication on 
\[
\bigoplus_{m > 0} \CF(\Ls{0},\Ls{m};\underline{\mathfrak{m}})
\]
where $\mathfrak{m} = \{x \in C : |x| < 1\}$.

\subsection{The Fargues-Fontaine graded ring}
\label{intro:FFgr}
A perfect field of characteristic $p$, complete with respect to a norm $|\cdot|$, is since \cite{Scholze} known as a ``perfectoid field of characteristic $p$.''  Suppose $(C,|\cdot|)$ is such a field, and suppose furthermore that $C$ is algebraically closed.  Let $E = \mathbf{F}_p(\!(z)\!)$, and let $B \supset E$ be the set of bi-infinite formal series $\sum_{i \in \bZ} b_i z^i \in \prod_{i \in \bZ} C z^i$ with coefficients $b_i \in C$, and which obey
\begin{equation}
\label{eq:FFKS}
\forall r \in (0,1) \qquad |b_i| r^i \to 0 \text{ as } |i| \to \infty
\end{equation}
This ring $B$ coincides with what is called $B_{(0,1)}$ in \cite[Ex. 1.6.5]{FF}, and what is called $\cO_{\bR^1}((0,1))$ in \cite[Def. 21]{KS}.  

The automorphism $\varphi:B \to B$ given by
\begin{equation}
\label{eq:phiBB}
\varphi:\left( \sum c_i z^i\right) \mapsto \sum c_i^p z^i \qquad (\text{i.e. } \varphi(f(z)) = f(z^{1/p})^p)
\end{equation}
cuts $B$ into ``eigenspaces''
$
B^{\varphi = z^n} := \left\{f \in B \mid \varphi(f) = z^n f \right\}
$.  The Fargues-Fontaine curve attached to $(E,C)$ is
\begin{equation}
\label{eq:FF-curve}
\FF_E(C):=\Proj\left(\bigoplus_{n = 0}^{\infty} B^{\varphi = z^n}\right)
\end{equation}
\begin{theorem*}
\eqref{eq:432p} is isomorphic as a graded-ring-without-unit to the irrelevant ideal of \eqref{eq:FF-curve}, i.e.
\begin{equation}
\label{eq:feste}
\bigoplus_{n = 1}^{\infty} \CF(\Ls{0},\Ls{n};\underline{\mathfrak{m}}) \cong \bigoplus_{n = 1}^{\infty} B^{\varphi = z^n}
\end{equation}
\end{theorem*}

\subsection{Annuli}
\label{intro:degzeropiece}
The degree zero piece $B^{\varphi = 1}$ of the Fargues-Fontaine graded ring \eqref{eq:FF-curve} is isomorphic to $E$ ($= \bF_p(\!(z)\!)$ in our case).  The theorem \eqref{eq:feste} does not explain how this part arises Floer-theoretically.  As in \S\ref{intro:HFL0L0}, it should come from the Floer cohomology of $\Ls{0}$ against itself --- a version of Floer cohomology with the $F$-field turned on --- but the usual rules for making sense of the nontransverse intersection $\Ls{0} \cap \Ls{0}$ have to be revisited when $t = 1$.

As we mentioned in \S\ref{intro:HFL0L0}, and review in \S\ref{subsec:cont}, the usual rules involve choosing a Hamiltonian isotopy $\{\phi^s\}_{s \in \bR}$ so that $\phi^s \Ls{0}$ and $\Ls{0}$ do meet transversely.  The problem that we encounter is that the quasi-isomorphism type of $\CF(\phi^s \Ls{0},\Ls{0};\underline{C})$, with its bigon differential, is no longer independent of $\phi$.  One still has natural maps between cochain complexes for different $\phi$, but they are not quasi-isomorphisms: the usual formula for the necessary cochain homotopies does not converge.

This is a well-known problem with ``convergent power series Floer cohomology.''  It is discussed in print in \cite[p.3]{Cho-Oh} and \cite[\S 4.2]{Auroux-annuli}, and perhaps elsewhere, but there is not much theory available for addressing it.  Still, we take the following point of view (which is only heuristic):
\begin{quote}
``Continuation'', i.e. the independence of the Hamiltonian displacement $\phi$, fails because there are pseudoholomorphic annuli in $T$ that have one side on $\phi^s \Ls{0}$ and the other side on $\Ls{0}$.
\end{quote}
For instance, Oszv\'ath and Szab\'o stick to ``admissible'' Heegaard diagrams to avoid problems with annuli like these \cite[\S 4.2.2]{OsSz}.  The problem they pose in Lagrangian Floer theory is closely related to the problem that closed gradient orbits pose in circle-valued Morse theory \cite{HutchingsLee}.  There is some speculation about incorporating them directly into Floer-theoretic invariants in \cite{Auroux-annuli}.

\subsection{Loud Floer cochains}
If $L$ and $L'$ are in different homology classes there are no annuli between them.  But there are infinitely many annuli between $\phi^s \Ls{0}$ and $\Ls{0}$ for any $\phi$.  If we fix an ``autonomous'' $\phi$, then we can get these under control by considering larger and larger $s$: for $s$ large, all of the annuli between $\phi^s \Ls{0}$ and $\Ls{0}$ have large area.  For instance:
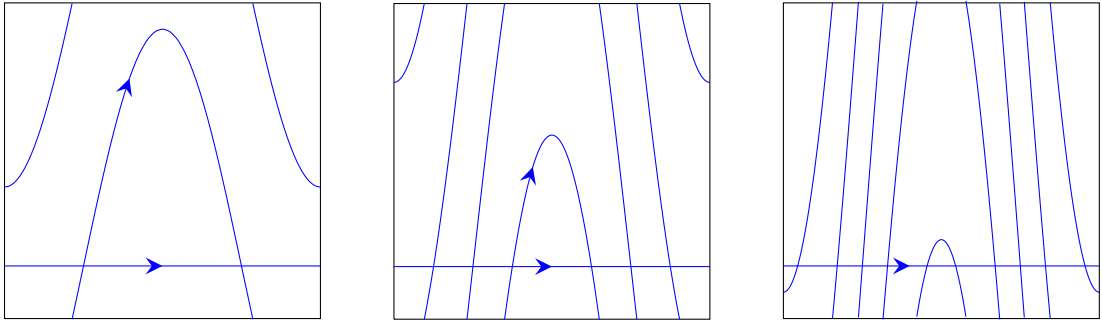
\begin{center}
\begin{tikzpicture}[scale=0.7]
      \tikzset{->-/.style={decoration={ markings,
        mark=at position #1 with {\arrow[scale=2,>=stealth]{>}}},postaction={decorate}}}
         \draw (0,-1)--(6,-1)--(6,6-1)--(0,6-1)--(0,-1);
         \draw[blue, ->-=.5]  (0,0) -- (6,0);        
         \draw[domain=0:77,smooth,variable=\x,blue] plot ({\x/60},{-4.5*cos(\x)+6});
         \draw[domain=77:283,smooth,variable=\x,blue, ->-=.4] plot ({\x/60},{-4.5*cos(\x)});
         \draw[domain=283:360,smooth,variable=\x,blue] plot ({\x/60},{-4.5*cos(\x)+6});
\end{tikzpicture}
\qquad
\begin{tikzpicture}[scale=0.7]
      \tikzset{->-/.style={decoration={ markings,
        mark=at position #1 with {\arrow[scale=2,>=stealth]{>}}},postaction={decorate}}}
         \draw (0,-1)--(6,-1)--(6,6-1)--(0,6-1)--(0,-1);
         \draw[blue, ->-=.5]  (0,0) -- (6,0);
         \draw[domain=0:34.56,smooth,variable=\x,blue] plot ({\x/60},{-8.5*cos(\x)+12});
         \draw[domain=34.56:83.24,smooth,variable=\x,blue] plot ({\x/60},{-8.5*cos(\x)+6});
         \draw[domain=83.24:126.02,smooth,variable=\x,blue] plot ({\x/60},{-8.5*cos(\x)});
         \draw[domain=126.02:234,smooth,variable=\x,blue, ->-=.4] plot ({\x/60},{-8.5*cos(\x)-6});
          \draw[domain=234:276.76,smooth,variable=\x,blue] plot ({\x/60},{-8.5*cos(\x)});
          \draw[domain=276.76:325.44,smooth,variable=\x,blue] plot ({\x/60},{-8.5*cos(\x)+6});
          \draw[domain=325.44:360,smooth,variable=\x,blue] plot ({\x/60},{-8.5*cos(\x)+12});        
\end{tikzpicture}
\qquad
\begin{tikzpicture}[scale=0.7]
      \tikzset{->-/.style={decoration={ markings,
        mark=at position #1 with {\arrow[scale=2,>=stealth]{>}}},postaction={decorate}}}
         \draw (0,-1)--(6,-1)--(6,6-1)--(0,6-1)--(0,-1);
         \draw[blue, ->-=.4]  (0,0) -- (6,0);
         \draw[domain=0:56,smooth,variable=\x,blue] plot ({\x/60},{-12.5*cos(\x)+12});
         \draw[domain=56:85.5,smooth,variable=\x,blue] plot ({\x/60},{-12.5*cos(\x)+6});
         \draw[domain=85.5:113.5,smooth,variable=\x,blue] plot ({\x/60},{-12.5*cos(\x)});
         \draw[domain=113.5:152,smooth,variable=\x,blue] plot ({\x/60},{-12.5*cos(\x)-6});
         \draw[domain=152:208,smooth,variable=\x,blue] plot ({\x/60},{-12.5*cos(\x)-12});
          \draw[domain=208:246.5,smooth,variable=\x,blue] plot ({\x/60},{-12.5*cos(\x)-6});
          \draw[domain=246.5:274.5,smooth,variable=\x,blue] plot ({\x/60},{-12.5*cos(\x)});
          \draw[domain=274.5:304,smooth,variable=\x,blue] plot ({\x/60},{-12.5*cos(\x)+6});
          \draw[domain=304:360,smooth,variable=\x,blue] plot ({\x/60},{-12.5*cos(\x)+12});        
\end{tikzpicture}
\end{center}
Some of the constructions of \cite{HLee} have inspired us, here.
The complexes $\CF(\phi^s \Ls{0},\Ls{0};\underline{C})$ for different $s$ do not all have the same cohomology but there are natural cochain maps
\[
\CF(\phi^s \Ls{0},\Ls{0};\underline{C}) \to \CF(\phi^{s'} \Ls{0},\Ls{0};\underline{C})
\]
whenever $s' > s$.  We will study the colimit of this filtered diagram.  For large $s$ the picture of $\phi^s \Ls{0}$ is a sine wave with large amplitude (wrapped up around the torus), we call
\begin{equation}
\label{eq:intro:cfloud}
\CF_{\mathrm{loud}}(\Ls{0},\Ls{0};\underline{C}) := \varinjlim_s \CF(\phi^s \Ls{0},\Ls{0};\underline{C})
\end{equation}
the loud Floer cochains on $(\Ls{0},\Ls{0})$.  The name was suggested to us by Johnson-Freyd.  Now our point of view is the following:
\begin{quote}
By shouting infinitely loud, all of the annuli break, along with whatever problems they posed for noninvariance.
\end{quote}
We will not try to make this precise, but for a somewhat analogous precedent in the setting of periodic orbits, see what is called the ``Latour trick'' in \cite{Hutchings}.  The Latour trick breaks up the periodic orbits of a closed $1$-form by adding a large multiple of an exact $1$-form.  One could equivalently think of pushing the graph of the closed $1$-form, for a long time, by the Hamiltonian flow of a primitive for the exact form.  A large \emph{finite} multiple of the exact form suffices to break up all the periodic orbits, while in \eqref{eq:intro:cfloud} one has to pass to the limit, but maybe ``shouting loud'' is not a worse metaphor for one process than for the other.

We will show that the triangles with sides on $\Ls{0}, \phi^s \Ls{0}, \phi^{s+s'} \Ls{0}$ induce a multiplication on $\CF_{\mathrm{loud}}(\Ls{0},\Ls{0};\underline{C})$ and on $\HF_{\mathrm{loud}}(\Ls{0},\Ls{0};\underline{C})$.  Our construction of this multiplication is quite crude: a better analysis would follow the construction of an $A_{\infty}$-structure on wrapped Floer cochains \cite{Abouzaid-Seidel}, which we expect to apply here and give a richer structure.  But our computations give an isomorphism of rings
\begin{equation}
\label{eq:introCzz}
\HF^0_{\mathrm{loud}}(\Ls{0},\Ls{0};\underline{C}) \cong C^{\sigma}[z,z^{-1}]
\end{equation}
When $\sigma$ is the $p$th root map, this is the Laurent polynomial ring $\bF_p[z,z^{-1}]$, a dense subring of $\bF_p(\!(z)\!)$.
\medskip

One can similarly define $\HF_{\mathrm{loud}}(\Ls{m},\Ls{m};\underline{C})$, and a ring structure on it, for any $m \in \bQ$.  One gets the same answer \eqref{eq:introCzz} when $m$ is an integer.  For $m = d/r$, we expect but will not prove that $\HF^0_{\mathrm{loud}}$ is a dense subring of an $r^2$-dimensional division algebra over $\bF_p(\!(z)\!)$ whose invariant is $m+\bZ \in \bQ/\bZ$.   The indecomposable vector bundles on the Fargues-Fontaine curve are classified by $\bQ$ and those division algebras arise as their endomorphism rings \cite[Thm. 8.2.10]{FF}.

\subsection{Acknowledgments} We have benefitted from conversations, suggestions and correspondence from D. Auroux, B. Bhatt, A. Caraiani, K. Fukaya, M. Kim, K. Madapusi-Pera, B. Poonen and P. Seidel.
YL is partially supported by the Royal Society URF\textbackslash R\textbackslash180024. DT was supported by NSF-DMS 1811971.

\section{Some Floer-theoretic background}
\label{two}

In this section we review some of Floer theory, making what simplifications are possible when the target manifold is a 2d torus $T$.

\subsection{$J$-holomorphic polygons}
Let $D \subset \bC$ be the closed unit disk in the complex plane.  We denote by $D^{\circ}$ the open unit disk and $\partial D = D - D^{\circ}$ the boundary of $D$.  Let $\mathbf{z} = (z_0,\ldots,z_n)$ denote an ordered $(n+1)$-tuple of points in $\partial D$.  We require that the points of $\mathbf{z}$ are pairwise distinct and that the counterclockwise arc subtending $z_{i-1}$ and $z_i$ (or $z_n$ and $z_0$) does not contain any other point of $\mathbf{z}$.

Let $L_0,L_1,\ldots,L_n$ be an $(n+1)$-tuple of one-dimensional submanifolds of $T$, and let $(x_0,\ldots,x_n)$ be a tuple of points in $T$ with
\[
x_0 \in L_0 \cap L_n, \qquad x_1 \in L_1 \cap L_0, \qquad \ldots,\qquad x_n \in L_n \cap L_{n-1}
\]
We write $\cW(x_0,\ldots,x_n)$ for the set of pairs $(\mathbf{z},u)$ where $u:D \to T$ is a continuous map, smooth away from $\mathbf{z}$, that carries $z_i$ to $x_i$ and that maps the counterclockwise arc between $z_i$ and $z_{i+1}$ into $L_i$.  It carries a topology in which a sequence $(\mathbf{z}_i,u_i)$ converges $(\mathbf{z},u)$ if $\mathbf{z}_i \to \mathbf{z}$ in $(\partial D)^{\times (n+1)}$ and $u_i \to u$ in a suitable Sobolev space. 
\medskip

Fix an almost complex structure $J$ on $T$.  
A polygon $(\mathbf{z},u) \in \cW(x_0,\ldots,x_n)$ is called $J$-holomorphic if the differential of $u$ is $\bC$-linear on each tangent space of the interior $D^{\circ}$.  Write $\widetilde{\cM}(x_0,\ldots,x_n)\subset \cW(x_0,\ldots,x_n)$ for the subspace of $J$-holomorphic polygons.  The group $\mathrm{PSL}_2(\bR)$ of biholomorphisms of $D^{\circ}$ acts on $\widetilde{\cM}$ by reparametrization, and we denote the quotient by $\cM(x_0,\ldots,x_n)$.

\subsection{Transversely cut criteria}
\label{tcc}
Each connected component of $\cM(x_0,\ldots,x_n)$ is labeled by a nonnegative integer called the analytic index of the component (or of any map in the component).  In case the conditions that cut $\widetilde{\cM}$ out of $\cW$ are transverse in a sense that we will not review here, then each component is a topological manifold and the analytic index coincides with the dimension of this component.  A formula for this dimension is given below \eqref{eq:dim-mas}.  When the ``transversely cut'' condition is satisfied, we call the components of dimension zero rigid polygons.  

The ``transversely cut'' condition is satisfied whenever the $(L_0,\ldots,L_n)$ are in general position --- that is, whenever the $L_i$ are pairwise transverse and $L_i \cap L_j \cap L_k$ is empty.  On $T$ or another surface, this triple intersection condition can be relaxed \cite[Lem. 13.2]{Seidel-book}, for instance $\cM$ is transversely cut in a neighborhood of $u$ as soon as $u$ is not constant.  Even some constant maps $u$ are transversely cut, for instance:
If $(L_0,L_1,L_2)$ are pairwise transverse, then at any triple intersection point $x \in L_0 \cap L_1 \cap L_2$, the tangent lines $(T_x L_0,T_x L_1,T_x L_2)$ come in either clockwise or counterclockwise order.  A constant map $D \to L_0 \cap L_1 \cap L_2$ contributes to $\cM(x,x,x)$ if and only if they come in clockwise order
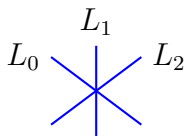
\begin{center}
\begin{tikzpicture}[scale = .6]
\tikzset{->-/.style={decoration={ markings,
        mark=at position #1 with {\arrow[scale=1.5,>=stealth]{>}}},postaction={decorate}}}
\draw[blue, thick] (-1,-.75)--(1,.75);
\draw[blue,thick] (1,-.75)--(-1,.75);
\draw[blue,thick] (0,-1)--(0,1);
\node[right] at (1,.75) {$L_2$};
\node[above] at (0,1) {$L_1$};
\node[left] at (-1,.75) {$L_0$};
\end{tikzpicture}
\end{center}
An example of a non-transversely cut quadrilateral in $T$ (necessarily constant) is analyzed in  \cite[Thm. 8]{LPshort}.  We will encounter some non-transversely cut triangles in \S\ref{subsec:penultimate}.  

\subsection{Maslov index of an intersection point}
\label{subsec:mas}
Let $(L,L')$ be an ordered pair of connected one-dimensional submanifolds of $T$, and fix an orientation of both $L$ and $L'$.  Suppose that $L$ and $L'$ meet transversely at the point $x$, then we define $\mas(x) \in \bZ/2$ by the rule indicated in the diagram
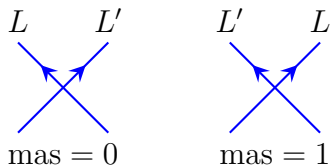
\begin{center}
\begin{tikzpicture}[scale = .6]
\tikzset{->-/.style={decoration={ markings,
        mark=at position #1 with {\arrow[scale=1.5,>=stealth]{>}}},postaction={decorate}}}
\draw[blue, thick, ->-=.75] (-1,-1)--(1,1);
\draw[blue,thick,->-=.75] (1,-1)--(-1,1);
\node[above] at (1,1) {$L'$};
\node[above] at (-1,1) {$L$};
\node[below] at (0,-1) {$\mas = 0$};
\end{tikzpicture}
\qquad
\begin{tikzpicture}[scale = .6]
\tikzset{->-/.style={decoration={ markings,
        mark=at position #1 with {\arrow[scale=1.5,>=stealth]{>}}},postaction={decorate}}}
\draw[blue, thick, ->-=.75] (-1,-1)--(1,1);
\draw[blue,thick,->-=.75] (1,-1)--(-1,1);
\node[above] at (1,1) {$L$};
\node[above] at (-1,1) {$L'$};
\node[below] at (0,-1) {$\mas = 1$};
\end{tikzpicture}
\end{center}
 If $L$ and $L'$ are not homologous to zero, one may lift the Maslov index to a $\bZ$-valued invariant by equipping $L$ and $L'$ (and $T$) with gradings, see \cite[\S 6]{LP} --- let us denote this $\bZ$-valued Maslov index by $\mas_{\bZ}(x)$.  A formula for the dimension near $u \in \cM(x_0,\ldots,x_n)$ is
\begin{equation}
\label{eq:dim-mas}
\mas_\bZ(x_0) - \mas_\bZ(x_1) - \cdots - \mas_\bZ(x_n)
\end{equation}
 
\subsection{Sign of a rigid polygon}
\label{sorp}
By making some additional choices one may attach a sign to each rigid polygon with boundary on $(L_0,\ldots,L_n)$ --- in other words one may define a map
\begin{equation}
\label{eq:sign-spin}
\cM(x_0,\ldots,x_n) \to \{1,-1\}
\end{equation}
We recall the recipe for \eqref{eq:sign-spin} given in \cite[\S 7]{Seidel} --- it depends on the choice of orientation for each $L_i$, and on the additional data of a basepoint $\star_i \in L_i$ in each $L_i$.  One requires that $\star_i \notin L_j$ for any $j \neq i$.  The point $\star_i$ endows $L_i$ with a nontrivial spin structure (also known as bounding or Neveu-Schwarz spin structure) which is trivialized away from $\star_i$.

If $u\vert_{\partial D}:\partial D \to \cup_{i = 0}^n L_i$ preserves the counter-clockwise orientation of $D$, the sign is $+1$ or $-1$ according to whether one encounters an even or odd number of stars going around $\partial D$ --- i.e. it is $(-1)^{\#u^{-1}\{\star_0,\ldots,\star_n\}}$.  Changing the orientation of $L_0$ does not change this sign, changing the orientation of $L_n$ changes the signs by $(-1)^{\mas(x_0) + \mas(x_n)}$, while changing the orientation of any of the other $L_i$ changes the sign by $(-1)^{\mas(x_i)}$.

For short, we will sometimes write
\begin{equation}
\label{eq:abs-mas}
(-1)^{|x|}:=(-1)^{\mas(x)}
\end{equation}

\subsection{Floer cochain complexes}
\label{subsec:tsfccam}
Let $t$ be a formal variable, and let $\bZ[t^{\bR_{\geq 0}}]$ denote the semigroup algebra of $\bR_{\geq 0}$, i.e. the group of finite $\bZ$-linear combinations of symbols of the form $t^a$, where $a \in \bR_{\geq 0}$, with the multiplication $t^a t^b = t^{a+b}$.  Let $\Lambda$ be a $\bZ[t^{\bR_{\geq 0}}]$-algebra.  In a moment we will take $\Lambda$ to be the Novikov completion of $\bZ[t^{\bR_{\geq 0}}]$ but the differential in $\CF(L,L')$ is given by a finite sum, so that in this section it might as well be $\bZ[t^{\bR_{\geq 0}}]$ itself.

When $L$ and $L'$ intersect transversely, then we let
\begin{equation}
\label{eq:CFLLp}
\CF(L,L') := \bigoplus_{x \in L \cap L'} \Lambda
\end{equation}
The choice of orientation for $L$ and $L'$ endows this group with a $\bZ/2$-grading
\[
\CF = \CF^0 \oplus \CF^1, \qquad \text{where }
\CF^i(L,L') := \bigoplus_{x | \mas(x) = i} \Lambda
\]
Further equipping $L$ and $L'$ with stars \S\ref{sorp} gives us the bigon differential:
\begin{equation}
\label{eq:mu1}
\mu_1:\CF^i(L,L') \to \CF^{i+1}(L,L'): x \mapsto \sum_{y \mid \mas(y) = i+1} y \left(\sum_{u \in \cM(x,y)} \pm t^{\area(u)}\right) 
\end{equation}
where the sign is given in \S\ref{sorp}, and $\mathrm{area}(u):=\int_D u^* (dx\, dy)$.  The inner sum is finite.  A general argument using non-rigid bigons shows that $\mu_1 \mu_1 = 0$ --- this is a case of the $A_{\infty}$-relations \S\ref{ainfrelations}.  In $T$, this can be proved more simply by lifting the grading from $\bZ/2$ to $\bZ$: the $\bZ$-grading is always concentrated in only two degrees.
\medskip

We have just described the ``absolute'' Floer cochain complex.  After fixing a point $D \in T$, not on $L$ or $L'$, we also have a ``relative to $D$''  complex in which $\Lambda$ in \eqref{eq:CFLLp} can be shrunk to $\bZ[t]$ (or another $\bZ[t]$-algebra), and the expression $t^{\mathrm{area}(u)}$ in is replaced by $t^{\#u^{-1}(D)}$.

\subsection{Example --- special Lagrangians}
\label{ex:sl}
We will call a circle $L \subset T$ a ``special Lagrangian'' if it is the image under $\bR^2 \to T$ of a straight line.  If that straight line has the form $y+mx = b$ then we will call $m$ the slope of the special Lagrangian, otherwise we say $L$ has slope $\infty$; thus the possible slopes are $m \in \bQ \cup \{\infty\}$.  If $L$ is special with finite slope, let us call the orientation under the parametrization $x \mapsto (x,b-mx)$ the ``default orientation.''

If $L \neq L'$ are two special Lagrangians, of finite slopes $m$ and $m'$, then they meet transversely in a set of cardinality $|nd' - n'd|$, if $m = n/d$ and $m' = n'/d'$.  All the intersection points are in a single Maslov degree; with the default orientations, these degrees are
\[
\CF(L,L') = \begin{cases}
\CF^0(L,L') & \text{if $m' > m$} \\
\CF^1(L,L') & \text{if $m' < m$}
\end{cases}
\]
There are no bigons and the differential \eqref{eq:mu1} is zero.  

\subsection{Polygon maps}
\label{two:polygon}
Suppose that $(L_0,\ldots,L_n)$ are in sufficiently general position that all the $\cM(x_0,\ldots,x_n)$ are transversely cut.  The $(n+1)$-gon map is a multilinear map
\begin{equation}
\label{eq:mun}
\mu_n:\CF^{i_n}(L_{n-1},L_n) \times \cdots \times \CF^{i_1}(L_0,L_1) 
\to \CF^{i_1+\cdots+i_n + 2 - n}(L_0,L_n)
\end{equation}
which carries $(x_n,x_{n-1},\ldots,x_1)$ to
\begin{equation}
\label{eq:mun-formula}
\sum_{y} y \left(\sum_{u \in \cM(y,x_1,x_2,\ldots,x_n)} \pm t^{\area(u)}\right)
\end{equation}
It is not defined until the $L_i$ are equipped with orientations and stars.  Furthermore, the inner sum in \eqref{eq:mun-formula} is usually infinite, so $\Lambda$ should carry a topology in which it converges.  The standard choice for $\Lambda$ is one of $\Lambda^0$ or $\Lambda^0[t^{-1}]$, where $\Lambda^0$ is the Novikov ring
\begin{equation}
\label{eq:LamZ0}
\Lambda^0 = \Lambda_{\bZ}^0 =\left\{\sum_{i = 0}^{\infty} a_i t^{\lambda_i} \mid a_i \in \bZ, \lambda_i \in \bR_{\geq 0} \text{ and } \lim_{i \to \infty} \lambda_i = \infty\right\}
\end{equation}
The fact that $\Lambda^0$ can be taken to have $\bZ$-coefficients is a reflection of the fact that the moduli spaces $\cM$ are not orbifolds --- this holds for $T$ and more generally for semipositive symplectic manifolds.
\medskip

In the relative setting, as long as $D$ does not lie on any $L_i$, we replace $t^\area(u)$ with $t^{\#u^{-1}(D)}$, and \eqref{eq:mun} is multilinear over $\bZ[\![t]\!]$.

\subsection{Example --- some triangle maps}
\label{two:tri-maps}
Suppose $L_0,\ldots,L_k$ are special Lagrangians of slopes
\begin{equation}
\label{eq:increasing-slope}
m_0 < m_1 < \cdots < m_k < \infty
\end{equation}
If $k \neq 2$, then the set of rigid $(k+1)$-gons with boundary on $L_0,\ldots,L_k$ is empty, and the maps
\[
\mu_k:\CF(L_{k-1},L_k) \times \cdots \times \CF(L_0,L_1) \to \CF(L_0,L_k)
\]
are zero --- this is a consequence of \S\ref{eq:dim-mas}.  In other words when \eqref{eq:increasing-slope} is satisfied $\mu_2$ is the only interesting polygon map.  In this section we explain how to compute $\mu_2$ in detail for Lagrangians of the form $\Ls{m_0}, \Ls{m_1},\Ls{m_2}$ \S\ref{intro:13}.  These are the maps that can be packed into a product structure on $\bigoplus \CF(\Ls{0},\Ls{m})$ \S\ref{intro:dehn},  and our notation is adapted to describing this product structure.  
\medskip

One consequence of the vanishing of the $\mu_k$ for $k \neq 2$ is that this product is strictly associative \S\ref{ainfrelations}, and we can describe the product in terms of the theta functions.  If \eqref{eq:increasing-slope} is not satisfied then the maps $\mu_k$ do not all vanish for $k \geq 3$ --- they can written in terms of Hecke's indefinite theta series \cite{Polishchuk}.
\medskip

If $m_1 < m_2$, there are $m_2 - m_1$ intersection points between $\Ls{m_1}$ and $\Ls{m_2}$, each contributing a basis element to $\CF(\Ls{m_1},\Ls{m_2})$.  Let us index those intersection points in the following way.
\begin{equation}
\tau^{m_1}(x_{m_2 - m_1,\kappa}) := (\kappa,-m_2\kappa)
\end{equation}
where $\kappa \in \{0,\frac{1}{m_2 - m_1},\ldots, \frac{m_2 - m_1 - 1}{m_2 - m_1}\}$ and $\tau$ denotes the Dehn twist map
\[
\tau:(x,y) \mapsto (x,y-x)
\]
Fix an irrational number $\varepsilon$ and equip each $\Ls{m}$ with a star \S\ref{sorp}
\begin{equation}
\label{eq:starm}
\star_{(m)} = \star_{(m),\varepsilon} := (\varepsilon,-m\varepsilon)
\end{equation}
The Dehn twist carries $\Ls{m}$ to $\Ls{m+1}$ and preserves the stars.  As $\varepsilon$ is irrational, the stars avoid the intersection points $\Ls{m_1} \cap \Ls{m_2}$.  We will compute the triangle maps
\[
\CF(\Ls{m_1},\Ls{m_1+m_2}) \times \CF(\Ls{0},\Ls{(m_1)}) \to \CF(\Ls{0},\Ls{m_1+m_2})
\]
by computing $\mu_2(\tau^{m_1} x_{m_2,\kappa_2}, x_{m_1,\kappa_1})$.  For example
\begin{equation}
\mu_2(\tau x_{1,0},x_{1,0}) = \begin{cases}
x_{2,0}\Big( \sum_{i \in \bZ} t^{i^2}\Big) + x_{2,\frac{1}{2}}\Big(\sum_{i \in \bZ} t^{(i+\frac{1}{2})^2}\Big) & \text{(absolute setting)} \\
x_{2,0}\Big( \sum_{i \in \bZ} t^{i^2}\Big) + x_{2,\frac{1}{2}}\Big(\sum_{i \in \bZ} t^{i(i+1)}\Big) & \text{(relative setting)}
\end{cases}
\end{equation}

\begin{theorem*}
Let
\begin{equation}
\label{eq:323}
E(\kappa_1,\kappa_2) = E_{m_1,m_2}(\kappa_1,\kappa_2) := \frac{m_1 \kappa_1 + m_2 \kappa_2}{m_1 + m_2}
\end{equation}
(which carries $\frac{1}{m_1} \bZ \times \frac{1}{m_2} \bZ$ into $\frac{1}{m_1+m_2} \bZ$). Then in the absolute setting $\mu_2(\tau^{m_1} x_{m_2,\kappa_2}, x_{m_1,\kappa_1})$ is given by
\begin{equation}
\label{eq:abs-mu2}
\mu_2(\tau^{m_1} x_{m_2,\kappa_2}, x_{m_1,\kappa_1}) 
= 
\sum_{\ell \in \bZ} x_{m_1 + m_2,E(\kappa_1,\kappa_2 + \ell)} t^{(\ell+\kappa_2 - \kappa_1)^2/(2(\frac{1}{m_1} + \frac{1}{m_2}))}
\end{equation}
Let the functions $\phi(s)$ and $\lambda(u,v) = \lambda_{m_1,m_2}(u,v)$ be given by (cf. \cite[p. 83]{LP})
\begin{equation}
\label{eq:Euv-lamuv}
\phi(s):= \lfloor s \rfloor s - \frac{1}{2} \lfloor s \rfloor \lfloor s+1 \rfloor; \qquad  \lambda(u,v) := m_1 \phi(u)  + m_2 \phi(v) - (m_1 + m_2) \phi(E(u,v))
\end{equation}
Then in the relative setting $\mu_2(\tau^{m_1} x_{m_2,\kappa_2}, x_{m_1,\kappa_1})$ is given by
\begin{equation}
\label{eq:rel-mu2}
\mu_2(\tau^{m_1} x_{m_2,.\kappa_2},x_{m_1,\kappa_1}) = \sum_{\ell \in \bZ} x_{m_1 + m_2,E(\kappa_1,\kappa_2 + \ell)} t^{\lambda(\kappa_1,\kappa_2+\ell)}
\end{equation}
where we understand $x_{m_1+m_2,E(\kappa_1,\kappa_2+\ell)} := x_{m_1 + m_2,\kappa}$ if $E(\kappa_1,\kappa_2+\ell) = \kappa$ modulo $1$.
\end{theorem*}

\begin{proof}
The index $\ell$ in either sum \eqref{eq:abs-mu2}, \eqref{eq:rel-mu2} determines a triangle, two of whose vertices are at $x_{m_1,\kappa_1}$ and $\tau^{m_1} x_{m_2,\kappa_2}$.  In the universal cover, the coordinates of all three vertices are
\[
(\kappa_1,0), \qquad (E(\kappa_1,\kappa_2+\ell),0), \qquad (\kappa_2+\ell,-m_1(\kappa_2 + \ell - \kappa_1))
\]
as in the diagram
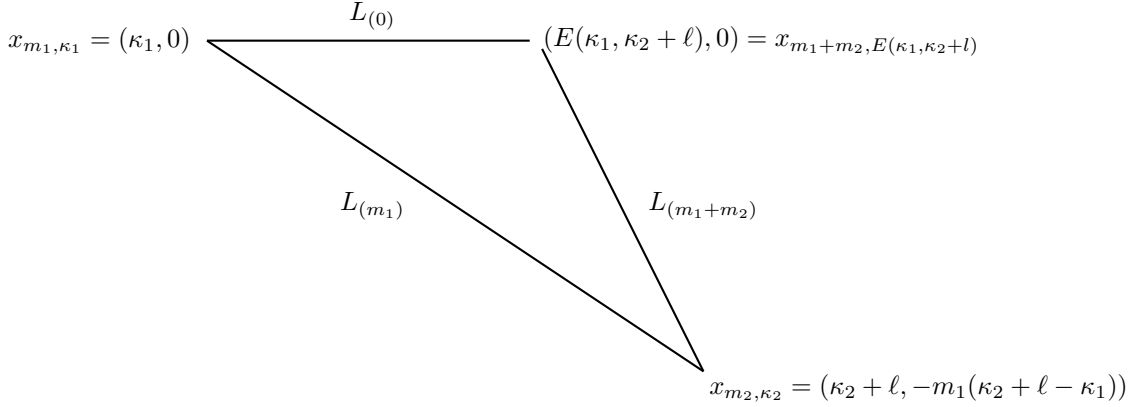
\begin{center}
\begin{tikzpicture} [scale=1.1]
      \tikzset{->-/.style={decoration={ markings,
        mark=at position #1 with {\arrow[scale=2,>=stealth]{>}}},postaction={decorate}}}
        
        \node at (-2,6.3) {\footnotesize $\Ls{0}$};
        \node at (2,4) {\footnotesize $\Ls{m_1+m_2}$};
        \node at (-2,4) {\footnotesize $\Ls{m_1}$};

        \draw[thick, black] (-4,6) -- (-.1,6);
	\node at (-5.3,6) {\footnotesize $x_{m_1,\kappa_1}= (\kappa_1, 0)$};
        \draw[thick, black] (0+.05,6-.1) -- (2,2);
         \node at (4.6,1.8) {\footnotesize 
	 $x_{m_2,\kappa_2}= \left( \kappa_2+\ell , -m_1(\kappa_2+\ell-\kappa_1) \right)$ };     
        \draw[thick, black] (-4,6) -- (2,2);
	\node at (2.7,6) {\footnotesize $(E(\kappa_1,\kappa_2+\ell), 0) = x_{m_1+m_2, E(\kappa_1,\kappa_2+l)}$};
\end{tikzpicture}
\end{center}
On any of these triangles there are an even number of stars \eqref{eq:starm}: for each $i \in \bZ$ with $\kappa_1 < \varepsilon + i < \kappa_2 + \ell$ there are exactly two stars whose $x$-coordinate (in the universal cover) is $\varepsilon+i$, one along $\Ls{m_1}$ and the other along either $\Ls{0}$ or $\Ls{m_1+m_2}$.  Thus every summand in the triangle has sign $+1$.

The exponent of $t$ in \eqref{eq:abs-mu2} is the area of the $\ell$th triangle, i.e.
\[ \frac{1}{2} m_1(\kappa_2+\ell-\kappa_1) (E(\kappa_1, \kappa_2+\ell) - \kappa_1) =(\ell+\kappa_2-\kappa_1)^2  / (2(\frac{1}{m_1}+\frac{1}{m_2})). \]
The more complicated exponent of $t$ in \eqref{eq:rel-mu2} is the lattice area \S\ref{intro:nov} of the same triangle, which coincides with the cardinality of $u^{-1}(D)$ when $D$ is in the first quadrant very close to $(0,0)$ ---  see \cite{LP}.
\end{proof}

\subsection{Theta functions}
\label{two:theta}
Let $\theta_{m,k}$, $\tabs_{m,k}$ be as in \eqref{eq:ttabs}
\[
\theta_{m,k} := \sum_{i = -\infty}^{\infty} t^{m \frac{i (i-1)}{2}+ k i}\, z^{mi + k}, \qquad \tabs_{m,k}  := \sum_{i = -\infty}^{\infty} t^{(mi+k)^2/(2m)} z^{mi+k}
\]
The Jacobi theta function is $\theta_{1,0}$, and the others are obtained by a simple change of variables
\[
\theta_{m,k}(t,z) = z^k \theta_{1,0}(t^m,t^k z^m), \qquad \theta_{m,k}(t,t^{\frac{1}{2}} z) = t^{{k(m-k)}/(2m)}\theta_{m,k}^{\mathrm{abs}}(t,z)
\]
Although these series are doubly infinite in $z$, when formally expanding the product of $\theta_{m,k}$ and $\theta_{m',k'}$ only finitely many terms contribute to the coefficient of any monomial $z^e t^f$ --- the same goes for $\tabs_{m,k}$ and $\tabs_{m',k'}$.  This is a consequence of the convexity of the functions $i \mapsto m \binom{i}{2}+ki$ and $i \mapsto (mi+k)^2/(2m)$ in the exponent of $t$.  That the resulting series $\theta_{m,k} \cdot \theta_{m',k'}$ or $\tabs_{m,k} \cdot \tabs_{m',k'}$ can be written as a linear combination of $\theta_{m+m',0},\cdots,\theta_{m+m',m+m'-1}$ is a standard but nontrivial fact about the theta functions, the formulas for these coefficients is the same as \eqref{eq:rel-mu2} and \eqref{eq:abs-mu2}: putting $\kappa_1 = k_1/m_1$ and $\kappa_2 = k_2/m_2$,
\[
\theta_{m_2,k_2} \cdot \theta_{m_1,k_1} = \sum_{\ell \in \bZ} \theta_{m_1 + m_2,E(\kappa_1,\kappa_2 + \ell)} t^{\lambda(\kappa_1,\kappa_2 + \ell)}
\] 
and
\[
\tabs_{m_2,k_2} \cdot \tabs_{m_1,k_1} = \sum_{\ell \in \bZ} \tabs_{m_1 + m_2,E(\kappa_1,\kappa_2 + \ell)} t^{(\ell+\kappa_2 - \kappa_1)^2/(2(\frac{1}{m_1} + \frac{1}{m_2}))}
\] 
In other words the map $x_{m,k} \mapsto \theta_{m,k}$ or $\tabs_{m,k}$ is a ring homomorphism.  One may verify this directly (and we will do so in the next section when we turn on an $F$-field), but it is natural to ask what is the Floer-theoretic origin of these series.  Each summand of \eqref{eq:ttabs} is indexed by a \emph{right} triangle, with one vertex at $x_{m,k/m}$ and sides along $\Ls{0},\Ls{m},\Ls{\infty}$.  The exponent of $t$ carries the area (or lattice area) of this right triangle and the exponent of $z$ carries the number of times the vertical edge of the triangle wraps around $\Ls{\infty}$ --- this $z$ can be interpreted as the monodromy of a rank one local system \S\ref{subs:localsystems}.  For instance the right triangles contributing to $\theta_{2,1}$ have the form (for $i$ positive)
\begin{center}
\begin{tikzpicture}
\tikzset{->-/.style={decoration={ markings,
        mark=at position #1 with {\arrow[scale=1.5,>=stealth]{>}}},postaction={decorate}}}
\draw[black, thick, ->-=.5] (0.1,0)--(1,0);
\draw[black,thick,->-=.5] (1,0)--(0,2);
\draw[black,thick,->-=.5] (0,2)--(0,.1);
\node[below right] at (1,0) {$x_{2,1/2}$};
\node[left] at (0,.75) {$z^{mi+ k} = z^{2i + 1}$};
\end{tikzpicture}
\end{center}

\subsection{The punctured torus, the large complex structure limit}
Part of the motivation for relative Floer theory is to make sense of the specialization $t = 0$ in the polygon sums.  Setting $t = 0$ has the effect of discarding the polygons that contribute a positive power of $t$, which in the relative case is the same as discarding the polygons that touch $D$.  One gets the same effect by doing Floer theory for Lagrangians in the punctured torus $T - D$.  When $T-D$ is equipped with the right symplectic structure, one with infinite area in a neighborhood of $D$, this is called the ``large volume limit'' of Floer theory.

We can treat the Lagrangians $\Ls{m}$ as boundary conditions for triangles $u:\Delta^2 \to T - D$, and sum over them to obtain a map $\mu_2$ as before, this time defined on the free $\bZ$-modules spanned by $\Ls{m_1} \cap \Ls{m_2}$.  For instance,
\[
\mu_2(\tau x_{1,0},x_{1,0}) = x_{2,0} + x_{2,1/2} \cdot 2
\]
where $x_{2,0}$ comes from a constant map and the two copies of $x_{2,1/2}$ come from the two shaded triangles in the figure 
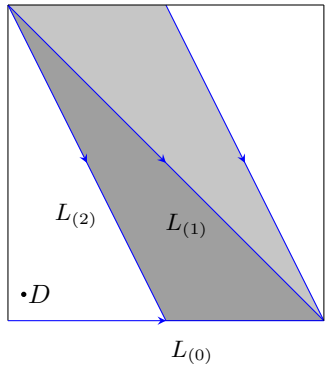
\begin{center}
\begin{tikzpicture}[scale = .7]
      \tikzset{->-/.style={decoration={ markings,
        mark=at position #1 with {\arrow[scale=1,>=stealth]{>}}},postaction={decorate}}}

   \begin{scope} 
          \clip (0,6) -- (3,0) -- (6,0) ;
           \clip[preaction={draw,fill=gray!75}] (0,0) rectangle (8,8);
   \end{scope}

  \begin{scope} 
          \clip (0,6) -- (3,6) -- (6,0) ;
           \clip[preaction={draw,fill=gray!45}] (0,0) rectangle (8,8);
   \end{scope}

	\draw (0,0) -- (0,6);
         
	 \draw[blue, ->-=.5]  (0,0) -- (6,0);
         \draw [blue, ->-=.5] (0,6) -- (3,0);
         \draw [blue, ->-=.5] (3,6) -- (6,0);
         \draw [blue, ->-=.5] (0,6) -- (6,0);

         \draw (0,6) -- (6,6);
         \draw (6,0) -- (6,6);

\node at (1.3,2)   {\scriptsize $\Ls{2}$};
\node at (3.4,1.8)   {\scriptsize $\Ls{1}$};

\node at (3.5,-0.6)   {\scriptsize $\Ls{0}$};

\draw[thick, fill=black] (0.3,0.5) circle(.03);

\node at (0.6,0.5) {\footnotesize $D$};

\end{tikzpicture}
\end{center}
We also record
\begin{eqnarray*}
\mu_2(\tau^2 x_{1,0},x_{2,0}) & = & x_{3,0} + x_{3,1/3} + x_{3,2/3} \\
\mu_2(\tau^2 x_{1,0},x_{2,1/2}) & = & x_{3,1/3} + x_{3,2/3} \\
\mu_2(\tau^3 x_{3,1/3},x_{3,2/3}) & = & x_{6,3/6} \\
\end{eqnarray*}
and
\[
\mu_2(\tau^4 x_{2,1/2}, \mu_2(\tau^2 x_{2,1/2},x_{2,1/2}))  =  x_{6,3/6}.
\]
If one sets $z = x_{1,0}$, $x = x_{2,1/2}$, $y = x_{3,2/3}$, these equation imply in particular that
\[
y^2 z + x^3 = xyz
\]
This is the equation for a nodal plane cubic curve --- the ``large complex structure limit'' that matches the large volume limit under mirror symmetry.

\subsection{$A_{\infty}$-relations}
\label{ainfrelations}
The moduli spaces $\cM$ that parametrize non-rigid polygons are not usually compact.  For example, suppose $L_0,L_1,L_2,L_3$ are as in the diagram
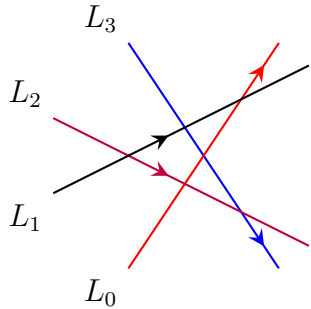
\begin{center}
\begin{tikzpicture}
\tikzset{->-/.style={decoration={ markings,
        mark=at position #1 with {\arrow[scale=1.5,>=stealth]{>}}},postaction={decorate}}}
\draw[blue, thick, ->-=.9] (0,1.5)--(2,-1.5);
\node[above left] at (0,1.5) {$L_3$};
\draw[red,thick,->-=.9] (0,-1.5)--(2,1.5);
\node[below left] at (0,-1.5) {$L_0$};
\draw[purple,thick,->-=.45] (-1,.5)--(2.4,-1.2);
\node[above left] at (-1,.5) {$L_2$};
\draw[black,thick,->-=.45] (-1,-.5)--(2.4,1.2);
\node[below left] at (-1,-.5) {$L_1$};

\end{tikzpicture}
\end{center}
Denote the red-black, black-purple, and purple-blue intersection points by $f$, $g$, and $h$ respectively, and the blue-red intersection point by $hgf$.  Then $\cM(hgf,f,g,h)$ includes the following one-parameter family of quadrilaterals:
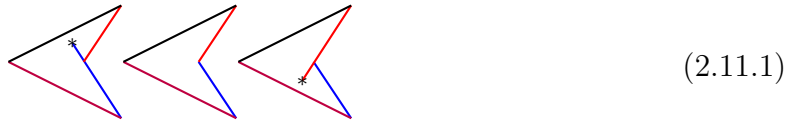
\begin{equation}
\label{eq:<<<}
\begin{aligned}
\begin{tikzpicture}
\tikzset{->-/.style={decoration={ markings,
        mark=at position #1 with {\arrow[scale=1.5,>=stealth]{>}}},postaction={decorate}}}
\draw[blue, thick] (1-.15,.25)--(1.5,-.75);
\node at (1-.15,.25) {\tiny $\ast$};
\draw[red,thick] (1,0)--(1.5,.75);
\draw[purple,thick] (0,0)--(1.5,-.75);
\draw[black,thick] (0,0)--(1.5,.75);
\end{tikzpicture}
\begin{tikzpicture}
\tikzset{->-/.style={decoration={ markings,
        mark=at position #1 with {\arrow[scale=1.5,>=stealth]{>}}},postaction={decorate}}}
\draw[blue, thick] (1,0)--(1.5,-.75);
\draw[red,thick] (1,0)--(1.5,.75);
\draw[purple,thick] (0,0)--(1.5,-.75);
\draw[black,thick] (0,0)--(1.5,.75);
\end{tikzpicture}
\begin{tikzpicture}
\tikzset{->-/.style={decoration={ markings,
        mark=at position #1 with {\arrow[scale=1.5,>=stealth]{>}}},postaction={decorate}}}
\draw[blue, thick] (1,0)--(1.5,-.75);
\draw[red,thick] (1-.15,-.25)--(1.5,.75);
\node at (1-.15,-.25) {\tiny $\ast$};
\draw[purple,thick] (0,0)--(1.5,-.75);
\draw[black,thick] (0,0)--(1.5,.75);
\end{tikzpicture}
\end{aligned}
\end{equation}
They all have the same image closure, but along the boundary may back-track along either the blue or the red line.  Near the $\ast$s, the map $u$ is biholomorphic to the map from the upper half-plane to $\bC$ that sends $z$ to $z^2$.  At the extreme parameters, where the $\ast$ reaches all the way to the black or to the purple line, there is no such $J$-holomorphic quadrilateral --- so $\cM$ is not compact --- but each of those extremes is occupied by a pair of $J$-holomorphic triangles.

Since the quadrilaterals are not rigid, they do not contribute to $\mu_3(h,g,f)$.  But the triangles at the extremes are rigid, at one end they contribute to $\mu_2(g,f)$ and $\mu_2(h,gf)$, and at the other end to $\mu_2(h,g)$ and $\mu_2(hg,f)$.  The interpolating family of quadrilaterals exhibits a relation between them.

More generally there is a compactification (the Deligne-Mumford-Stasheff compactification) of $\cM$.  When everything is transversely cut \S\ref{tcc}, the compactification is a topological manifold-with-corners, whose corners are indexed by tuples of rigid polygons.  Equipping the $L_i$ with orientations and stars induces an orientation on $\cM$ and its compactification --- \eqref{eq:sign-spin} is a special case of this orientation.  The oriented compactification of $\cM$ is used in the proof (it essentially is the proof) of the following equations among the polygon maps $\mu_n$:
\begin{equation}
\label{eq:ainfrelations}
\sum_{i+j = n+1} \sum_{\ell < i} (-1)^{|f_1|+\cdots+|f_\ell|-\ell} \mu_i(f_n,,\ldots,f_{\ell + j + 1},\mu_j(f_{\ell+j},\ldots,f_{\ell+1}),f_{\ell},\ldots,f_1) = 0
\end{equation}

The algebraic structure formed by the $\CF(L,L')$ and the maps $\mu_n$, subject to the relations \eqref{eq:ainfrelations}, is called an ``$A_{\infty}$-precategory'' in \cite[\S 4.3]{KS}.  Each $L$ is like an object, and each $f \in \CF(L,L')$ is like a morphism between objects.  \eqref{eq:ainfrelations} expresses the fact that these morphism spaces are cochain complexes and that the composition law is associative up to chain homotopy in a strong sense.  It falls short of being an $A_{\infty}$-category, for instance because $\CF(L,L')$ is defined only when $L$ and $L'$ meet transversely, so $\CF(L,L)$ is undefined and there is no ``morphism'' that could play the role of the identity map.  A standard way to address this problem is by analyzing the sense in which $\CF(L,L')$ are invariant under Hamiltonian isotopies --- Floer's theory of continuation.

\subsection{Example --- identity maps}
\label{ex-idmap}
There is a Floer cochain complex $\CF(L,L')$, well-defined up to quasi-isomorphism, even if $L$ and $L'$ do not intersect transversely.  In case $L'$ meets both $L_0$ and $L_1$ transversely, then in the absolute setting (resp. relative setting) any Hamiltonian isotopy (resp. any Hamiltonian isotopy supported on the complement of $D$) induces a quasi-isomorphism between $\CF(L_0,L')$ and $\CF(L_1,L')$ \S\ref{subsec:cont}.  One accesses $\CF(L,L')$ by perturbing $L$.

We illustrate this in an example that fills in the zeroth graded piece of \eqref{eq:121}.  Let $\phi^s$ denote the flow of $H(x,y) = \sin(2 \pi x)$, i.e.
\begin{equation}
\label{eq:phis}
\phi^s(x,y) = (x,y - s \cos(2\pi x))
\end{equation}
Then for $0 < s < 1$, $\phi^s \Ls{0}$ meets $\Ls{0}$ transversely at the points $x = (.25,0)$ and $y = (.75,0)$, so that $\CF(\phi^s \Ls{0},\Ls{0}) = \Lambda x \oplus \Lambda y$.  In a suitable fundamental domain the picture is this:
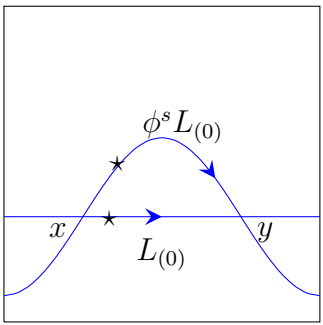
\begin{center}
\begin{tikzpicture}[scale=0.7]
      \tikzset{->-/.style={decoration={ markings,
        mark=at position #1 with {\arrow[scale=2,>=stealth]{>}}},postaction={decorate}}}
         \draw (0,-1)--(6,-1)--(6,6-1)--(0,6-1)--(0,-1);
         \draw[blue, ->-=.5]  (0,1) -- (6,1);
         \draw[domain=0:6,smooth,variable=\x,blue, ->-=.65] plot ({\x},{1-1.5*cos(60*\x)});
         \node at (2,0.97) {$\star$};
         \node at (2.15,2) {$\star$};
         \node[above left] at (1.4,0.4) {$x$};
         \node[above right] at (4.6,0.3) {$y$};
         \node[above] at (3.4,2.2) {$\phi^s \Ls{0}$};
         \node[below] at (3,0.8) {$\Ls{0}$};
\end{tikzpicture}
\end{center}
Equipping $\Ls{0}$ with its default orientation and $\phi^s \Ls{0}$ with the orientation induced by $\phi^s:\Ls{0} \cong \phi^s \Ls{0}$, the Maslov degrees are
\begin{equation}
\label{eq:herbert}
\CF^0(\phi^s \Ls{0},\Ls{0}) = x \Lambda \qquad \CF^1(\phi^s \Ls{0},\Ls{0}) = y\Lambda
\end{equation}
There are two bigons contributing to the differential $\mu_1$, of equal area $A$.  Both bigons have input $x$ and output $y$, and their signs \S\ref{sorp} are opposite to each other (no matter where the stars are placed), so that the differential is
\[
\mu_1(x) = y(t^A - t^A) = 0
\] 
Some variations of this computation are made in \S\ref{subsec:ex-parallel} and \S\ref{three:id-map}.

\subsection{Continuation}
\label{subsec:cont}
Let $X_H$ denote the Hamiltonian vector field of a function $H:[0,1] \times T \to \bR$, and write $\phi^s$ for its time $s$ flow, $\phi^s:T \to T$.  Let us review how the quasi-isomorphism
\begin{equation}
\label{eq:blobel}
\CF(L,L') \to \CF(\phi^s L, L')
\end{equation}
works in the absolute setting.  The map \eqref{eq:blobel} goes back to \cite[Thm. 4]{Floer}, our notation is closer to the appendix of \cite{Auroux}.  It is defined in terms of a set $\cM(x,y,\phi,\beta)$ of maps
\[
u:[-\infty,\infty] \times [0,1] \to T
\]
that obey the boundary conditions 
\[
\begin{array}{rcl}
u([-\infty,\infty] \times \{0\}) & \subset & L \\ \quad 
u([-\infty,\infty] \times \{1\}) & \subset & L'
\end{array}
\qquad
u(\{\infty\} \times [0,1]) = \{x\}
\qquad
u(-\infty,\tau) = \phi^{s \cdot \tau}(y)
\]
and (with analytic index zero \S\ref{tcc}) Floer's $X_H$-perturbed $J$-holomorphic curve equation:
\begin{equation}
\label{eq:aston}
\partial u/\partial \sigma + J(\partial u/\partial \tau - \beta(\sigma) s X_H) = 0
\end{equation}
Here $\beta$ (the ``profile function'') is a monotone decreasing $\bR$-valued function on $[-\infty,\infty]$ with $\beta(\sigma) = 1$ for $\sigma \ll 0$ and $\beta(\sigma) = 0$ for $\sigma \gg 0$.  The formula for \eqref{eq:blobel} is
\begin{equation}
\label{eq:palade}
x \mapsto \sum_{y \in \phi^s L \cap L'} y \sum_{u \in \cM(x,y,\phi)} \pm t^{\text{topological energy of $u$}}
\end{equation}
where the ``topological energy'' is (see Lemma 14.4.5 \cite{Ohvol2}).
\[ \int u^* \omega + \int_{0}^{1} H(\tau, u(\infty, \tau)) d\tau  - \int_{-\infty}^{\infty} \beta'(\sigma) s \int_0^1 (H_\tau \circ u) d\tau d\sigma  \]
This is sometimes a negative quantity, so we must allow $t^{-1} \in \Lambda$.
\medskip

The homotopy inverse to \eqref{eq:blobel} is just the continuation map for the reversed flow $\phi^{-s}$.  To describe the chain homotopy between the composite
\begin{equation}
\label{eq:con-con-inverse}
\CF(\phi^s L,L') \to \CF(\phi^{-s} \phi^s L,L') = \CF(L,L') \to \CF(\phi^s L,L')
\end{equation}
and the identity, map, let $B_r(\sigma)$ (for each $r > 0$) be a function that vanishes on an interval of length $r$ and that agrees up to translation of $\sigma$ with $\beta(\sigma)$ when $\sigma$ is to the left of that interval and with $\beta(-\sigma)$ when $\sigma$ is to the right of that interval.  Then
\begin{equation}
\label{eq:kossel}
\Delta(y) = \sum_x x \sum_u \pm t^{\text{topological energy of $u$}}  
\end{equation}
where the inner sum is over strips $u:[-\infty,\infty] \times [0,1] \to T$ that solve (for some $r$, with index $-1$)
\[
\partial u/\partial \sigma + J(\partial u/\partial \tau - B_r(\sigma) s X_H) = 0 \qquad r \in \bR_{\geq 0}
\]
and that have $u(-,0) \subset L$, $u(-,1) \subset L'$, and $u(-\infty,\tau) = \phi^{s\cdot \tau}(x)$ and $u(\infty,\tau) = \phi^{s\cdot \tau}(y)$.

\subsection{Example}
\label{subsec:ex-parallel}
Suppose that $L$ and $L'$ are a pair of parallel, horizontal circles, at distance $c$ apart.  With $\phi^s$ as in \eqref{eq:phis}, $\phi^s(L) \cap L'$ is empty unless $|s| > c$.  If $|s|$ only slightly exceeds $c$, then $\phi^s(L) \cap L'$ has two intersection points, say $x$ and $y$ as in the diagram:

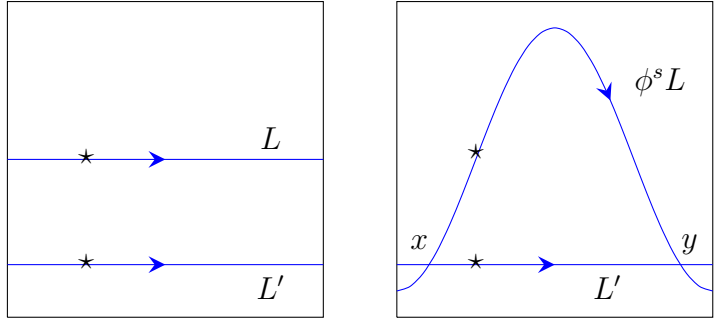
\begin{center}
\begin{tikzpicture}[scale=0.7]
      \tikzset{->-/.style={decoration={ markings,
        mark=at position #1 with {\arrow[scale=2,>=stealth]{>}}},postaction={decorate}}}
         \draw (0,-1)--(6,-1)--(6,6-1)--(0,6-1)--(0,-1);
         \draw[blue, ->-=.5]  (0,0) -- (6,0);
         \draw [blue, ->-=.5] (0,2) -- (6,2);
         \node at (1.5,0.05) {$\star$};
         \node at (1.5,2.05) {$\star$};
         \node[above] at (5,2) {$L$};
         \node[below] at (5,0) {$L'$};
\end{tikzpicture}
\qquad
\begin{tikzpicture}[scale=0.7]
      \tikzset{->-/.style={decoration={ markings,
        mark=at position #1 with {\arrow[scale=2,>=stealth]{>}}},postaction={decorate}}}
         \draw (0,-1)--(6,-1)--(6,6-1)--(0,6-1)--(0,-1);
         \draw[blue, ->-=.5]  (0,0) -- (6,0);
         \draw[domain=0:6,smooth,variable=\x,blue, ->-=.65] plot ({\x},{2-2.5*cos(60*\x)});
         \node at (1.5,0.05) {$\star$};
         \node at (1.5,2.15) {$\star$};
         \node[above left] at (0.8,.1) {$x$};
         \node[above right] at (5.2,0) {$y$};
         \node[above] at (5,3) {$\phi^s L$};
         \node[below] at (4,0) {$L'$};
\end{tikzpicture}
\end{center}
Thus, $\CF(L,L') = 0$, while (noting that there are two bigons in the right picture, a small one of area $A$ and a large one of area $A+c$) and $\CF(\phi^s L, L') = x \Lambda \oplus y\Lambda$, with differential
\[
\mu_1(x) = y (t^{A} - t^{A+c})
\]
(see \S\ref{sorp} for the signs).

As $\CF(L,L') = 0$, the composite \eqref{eq:con-con-inverse} is zero.  By Floer's theorem, the identity map on $\CF(\phi^s L,L')$ is chain homotopic to zero, with \eqref{eq:kossel} supplying the contracting homotopy.  Indeed the contracting homotopy is the geometric series
\[
\Delta(y) = x \cdot (t^{-A} + t^{-A + c} + \cdots + t^{-A + nc}  + \cdots)
\]
The strip $u_n$ contributing the $n$th term in the series (of topological energy $nc - A$) stretches horizontally across $n+2$ fundamental domains, crossing the boundary of the fundamental domain exactly $n+1$ times.  Here is a picture of $u_0\vert_{\partial ([-\infty,\infty] \times [0,1])}$ and $u_1\vert_{\partial ([-\infty,\infty] \times [0,1])}$ 
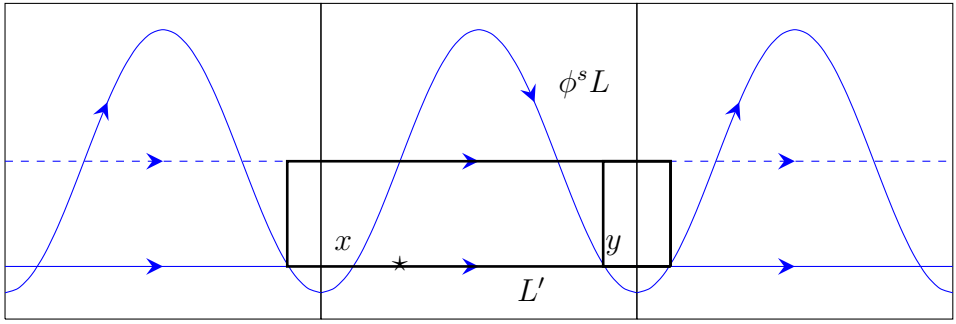
\begin{center}
 \begin{tikzpicture}[scale=0.7]
      \tikzset{->-/.style={decoration={ markings,
        mark=at position #1 with {\arrow[scale=2,>=stealth]{>}}},postaction={decorate}}}
         \draw (0,-1)--(6,-1)--(6,6-1)--(0,6-1)--(0,-1);
         \draw[blue, ->-=.5]  (0,0) -- (6,0);
         \draw[domain=0:6,smooth,variable=\x,blue, ->-=.65] plot ({\x},{2-2.5*cos(60*\x)});
         \node at (1.5,0.05) {$\star$};
        
         \node[above left] at (0.8,.1) {$x$};
         \node[above right] at (5.2,0) {$y$};
         \node[above] at (5,3) {$\phi^s L$};
         \node[below] at (4,0) {$L'$};
          \draw [blue, ->-=.5,dashed] (0,2) -- (6,2);
          \draw (0-6,-1)--(6-6,-1)--(6-6,6-1)--(0-6,6-1)--(0-6,-1);
          \draw[domain=0-6:6-6,smooth,variable=\x,blue, ->-=-.65] plot ({\x},{2-2.5*cos(60*\x)});
          \draw[blue, ->-=.5]  (0-6,0) -- (6-6,0);
          \draw [blue, ->-=.5,dashed] (0-6,2) -- (6-6,2);
          
          \draw (0+6,-1)--(6+6,-1)--(6+6,6-1)--(0+6,6-1)--(0+6,-1);
          \draw[domain=0+6:12,smooth,variable=\x,blue, ->-=-.65] plot ({\x},{2-2.5*cos(60*\x)});
          \draw[blue, ->-=.5]  (0+6,0) -- (6+6,0);
          \draw [blue, ->-=.5,dashed] (0+6,2) -- (6+6,2);

          \draw[black,line width = 1pt] (6+.64,0)--(6+0.64,2)--(6-.64,2)--(6-.64,0)--(6+.64,0);
          
          \draw[black,line width = 1pt] (6+.64,0)--(6+0.64,2)--(0-.64,2)--(0-.64,0)--(6+.64,0);

\end{tikzpicture}
\end{center}

\section{Floer theory coupled to an $F$-field}
\label{three}

In this section we go over the constructions and calculations of \S\ref{two}, coupling all the sums over polygons to a sheaf of rings $\underline{\Lambda}$ of the kind discussed in \S\ref{intro:ffield}.  We are interested in the case when $\underline{\Lambda}$ is pulled back along the projection map $\mathfrak{f}:T \to S^1$ \eqref{eq:intro:ffield}, and set up some notation for dealing with that case in \S\ref{three:Ffield}.

\subsection{Cochain complex}
If $L$ and $L'$ are two embedded circles in $T$, meeting transversely, let us write (just as in \eqref{eq:intro:172})
\begin{equation}
\label{eq:421}
\CF(L,L';\underline{\Lambda}) := \bigoplus_{x \in L \cap L'} \underline{\Lambda}_x
\end{equation}
If $a \in \underline{\Lambda}_x$ and we wish to regard it as an element of \eqref{eq:421}, we will write it  as $x \cdot a$.  We equip \eqref{eq:421} with a $\bZ/2$-grading by equipping $L$ and $L'$ with orientations, just as in \S\ref{subsec:tsfccam}.  After further equipping $L$ and $L'$ with stars, we define a differential
\[
\mu_1:\CF^i(L,L';\underline{\Lambda}) \to \CF^{i+1}(L,L';\underline{\Lambda})
\]
by the following analog of \eqref{eq:mu1}:
\begin{equation}
\label{eq:423}
x\cdot a \mapsto \sum_{y | \mas(y) = i+1}  y \left( \sum_{u \in \cM(y,x)} \pm t^{\mathrm{area}(u)} \nabla \gamma'\left(a \nabla \gamma(1_{\underline{\Lambda}_y})\right)\right)
\end{equation}
where
\begin{itemize}
\item $\gamma:I \to L$ is the path along the $L$-side of the bigon $u$ starting at $y$ and ending at $x$,
\item $\gamma':I \to L'$ is the path along the $L'$-side of $u$ starting at $x$ and ending at $y$.
\end{itemize}
This differential does not obey anything like $\mu_1(xa ) = \mu_1(x)a$ --- in fact $ \mu_1(x)a$ is typically undefined and in general $\mu_1(x a)$ and $\mu_1(x)$ do not have any useful relationship with each other.
\medskip

We obtain a $\underline{\Lambda}$-version of the continuation map \eqref{eq:blobel}
\begin{equation}
\label{eq:blobel-2}
\CF(L,L';\underline{\Lambda}) \to \CF(\phi^s L,L';\underline{\Lambda})
\end{equation}
by multiplying each summand of \eqref{eq:palade} by $\nabla \gamma'(a \nabla \gamma(\nabla(\phi^{-\tau s}(y))(1_{\underline{\Lambda}_y})))$, where $\gamma = u(\tau,0)$ and $\gamma' = u(-\tau,1)$ are the paths along $L$ and $L'$ respectively, and $\phi^{-\tau s}(y)$ is the reverse of the trajectory from $y$ to $\phi^s y$, i.e.
\begin{equation}
\label{eq:cont-twist}
\sum_{y \in \phi^s L \cap L'} y \sum_{u \in \cM(x,y,\phi)} \pm t^{\text{topological energy of $u$}}
\nabla \gamma'(a \nabla \gamma(\nabla (\phi^{-\tau s}(y))(1_{\underline{\Lambda}_y})))
\end{equation}
The same recipe as \S\ref{subsec:cont} gives the homotopy inverse to \eqref{eq:blobel-2}, only replacing \eqref{eq:kossel} by 
\begin{equation}
\label{eq:delta-twist}
a \cdot y \mapsto \sum_x x \sum_u \pm t^{\text{topological energy of $u$}} \nabla \gamma' \left(a \cdot \nabla (\phi^{\tau s} (y)) \nabla \gamma \nabla( \phi^{-\tau s}(x)) (1_{\underline{\Lambda}_x})\right)
\end{equation}

\subsection{Polygon maps}
\label{subsec:44}
Let $L_0,L_1,\ldots,L_n$ be oriented, starred submanifolds of $T$ as in \S\ref{two:polygon}.  We define a variant of \eqref{eq:mun}
\begin{equation}
\label{eq:441}
\mu_n:\CF^{i_n}(L_{n-1},L_n;\underline{\Lambda}) \times \cdots  \times \CF^{i_1}(L_0,L_1;\underline{\Lambda}) 
\to \CF^{i_1+\cdots+i_n + 2 - n}(L_0,L_n;\underline{\Lambda})
\end{equation}
carrying $(x_n \cdot a_n,x_{n-1} \cdot a_{n-1},\ldots, x_1 \cdot a_1)$ (with each $a_i \in \underline{\Lambda}_{x_i}$) to
\begin{equation}
\label{eq:442}
\sum_y y \left( \sum_u \pm t^{\mathrm{area}(u)} \nabla \gamma_n(a_n \nabla \gamma_{n-1} (a_{n-1} \cdots \nabla \gamma_2(a_2 \nabla \gamma_1(a_1\nabla(\gamma_0(1))\cdots))))\right)
\end{equation}
where $u$ runs over the same set of rigid polygons as \eqref{eq:mun}, the signs are just the same, and $\gamma_i$ is the path along the boundary of $u$ going from $x_i$ to $x_{i+1}$, or from $x_{n-1}$ to $x_0$.

Each connected component of the Deligne-Mumford-Stasheff compactification of the space of non-rigid polygons has the same vertices $x_0,\ldots,x_n$ --- that is, every $u$ in that component has those same vertices.  Moreover, the path from $x_i$ to $x_{i+1}$ or from $x_n$ to $x_0$ along $u$ belongs to the same homotopy class, so that $\nabla(\gamma_i):\Lambda_{x_i} \to \Lambda_{x_{i+1}}$ is locally constant in $u$.  Thus the $A_{\infty}$-relations among the \eqref{eq:441} hold for the usual reasons:
\begin{equation}
\sum_{i+j = n+1} \sum_{\ell < i} (-1)^{|x_1| + \cdots+ |x_{\ell}| - \ell} \mu_i(x_n a_n,\ldots, x_{\ell+j+1}a_{\ell+j+1},\mu_j(\cdots), x_\ell a_\ell,\ldots  x_1 a_1) = 0
\end{equation}

\subsection{Local systems}
\label{subs:localsystems}
We can put the formulas in \S\ref{subsec:44} in context by considering local systems on the $L_i$.  For each $i$ let $\cE_i$ be a local system of $\underline{\Lambda}\vert_{L_i}$-modules on $L_i$.  Then we define
\begin{equation}
\label{eq:451}
\CF((L_i,\cE_i),(L_{i+1},\cE_{i+1});\underline{\Lambda}) = \bigoplus_{x \in L_i \cap L_{i+1}} \Hom(\cE_{i,x},\cE_{i+1,x})
\end{equation}
The differential is modified by
\begin{equation}
\label{eq:452}
\mu_1(xf:\cE_{i,x} \to \cE_{i+1,x}) = \sum y \sum_u \pm t^{\area(u)} \nabla \gamma' \circ f \circ \nabla \gamma
\end{equation}
(where $xf$ denotes $f$ placed in the $x$th summand of \eqref{eq:451}).  In case $\cE_i = \underline{\Lambda}\vert_{L_i}$ is the trivial sheaf of modules, then $\Hom(\cE_{i,x},\cE_{i+1,x}) = \Hom(\underline{\Lambda}_x,\underline{\Lambda}_x)$ is naturally identified with $\underline{\Lambda}_x$ and \eqref{eq:452} coincides with \eqref{eq:423}.

The polygon maps
\begin{multline}
\label{eq:453}
\mu_n:\CF((L_{n-1},\cE_{n-1}),(L_n,\cE_n);\underline{\Lambda}) \times \cdots \times \CF((L_0,\cE_0),(L_1,\cE_1);\underline{\Lambda}) \\
\to \CF((L_0,\cE_0),(L_n,\cE_n);\underline{\Lambda})
\end{multline}
are defined by sending $(f_n,f_{n-1},\ldots,f_1)$ to the formal expression
\begin{equation}
\label{eq:454}
\sum_{y} y \left(\sum_{u \in \cM(y,x_1,x_2,\ldots,x_n)} \pm t^{\area(u)}\nabla \gamma_n\circ f_n\circ \nabla\gamma_{n-1} \circ f_{n-1} \circ \cdots \circ \nabla \gamma_1 \circ f_1 \circ \nabla \gamma_0 \right)
\end{equation}
We've left the degrees in \eqref{eq:453} off for typesetting reasons; they are the same as in \eqref{eq:mun}.  The formula \eqref{eq:454} specializes to \eqref{eq:442} in case $\cE_i = \underline{\Lambda}\vert_{L_i}$.  In general \eqref{eq:454} can fail to converge, unless the following ``unitarity'' condition is imposed on the $\cE_i$:
\begin{quote}
Each fiber of $\cE_i$ is locally free of finite rank over $\underline{\Lambda}\vert_{L_i}$, and the monodromy preserves an $\underline{\Lambda}^0\vert_{L_i}$-lattice.
\end{quote}

\subsection{$F$-field}
\label{three:Ffield}
We would like to package the $\underline{\Lambda}$-coupled triangle products among the $\Ls{m}$ into a graded ring, as in \S\ref{intro:dehn}.  The necessary natural isomorphism between $\CF(\Ls{m},\Ls{n};\underline{\Lambda})$ and $\CF(\Ls{0},\Ls{n-m};\underline{\Lambda})$ exists only when $\underline{\Lambda}$ is pulled back along the projection map
\begin{equation}
\label{eq:191}
\mathfrak{f}:T \to S^1 :(x,y) + \bZ^2 \mapsto x + \bZ
\end{equation}
To make this explicit, let $\sigma:C \to C$ be a ring automorphism, where $C$ is commutative.  We also let $\sigma$ denote induced automorphism of $C[\![t]\!]$ or of $\Lambda_C$, with $\sigma(t^a) = t^a$.  We are mainly interested in the case that
\begin{equation}
\label{eq:192}
\text{$C$ perfect of characteristic $p$}, \quad \sigma(c) = c^{1/p}
\end{equation}
in which case if $f(t) = \sum c_a t^a$ belongs to $C[\![t]\!]$ or to $\Lambda_C$ then we can write $\sigma(f)(t) = f(t^p)^{1/p}$.  The quotient
\begin{equation}
\label{eq:etale-space}
(\bR \times C)/\sim
\end{equation}
of $\bR \times C$ by the equivalence relation $(x,c) \sim (x+1,\sigma(c))$ is the \'etal\'e space of a locally constant sheaf of rings on $S^1$ --- as are $(\bR \times C[\![t]\!])/\sim$ and $(\bR \times \Lambda_C)/\sim$.  We denote the pullback-to-$T$ of these sheaves of rings by $\underline{C}$, $\underline{C}[\![t]\!]$, and $\Lambda_{\underline{C}}$.

We will call \eqref{eq:191} an ``$F$-field'' on $T$.  In diagrams, we keep track of it with a red line --- the inverse image of a point close to the right edge of the fundamental domain $[0,1)$ of $S^1$, as in the figure on the left below.  On the right we have drawn, in a different scale, the preimage of the red line in part of the universal cover of $T$, along with a triangle that contributes to $\mu_2(x_2 \cdot b,x_1 \cdot a)$.  One understands that $\sigma$ or $\sigma^{-1}$ is to be applied every time one crosses this ``danger line'' --- $\sigma$ if one crosses it from right to left, $\sigma^{-1}$ if one crosses it from left to right.  
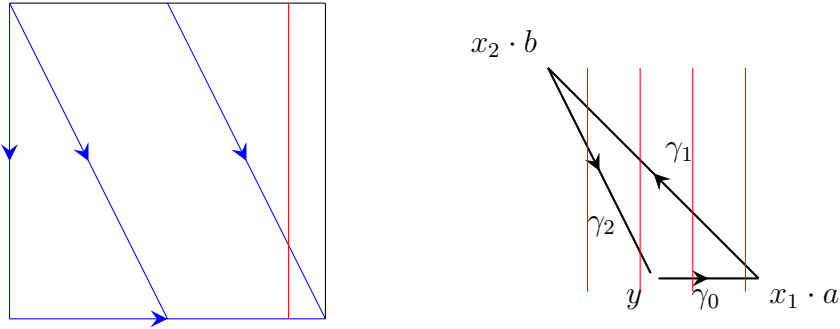
\begin{center}
\begin{tikzpicture}[scale=0.7]
      \tikzset{->-/.style={decoration={ markings,
        mark=at position #1 with {\arrow[scale=2,>=stealth]{>}}},postaction={decorate}}}
         \draw[blue, ->-=.5]  (0,0) -- (6,0);
         \draw [blue, ->-=.5] (0,6) -- (0,0);
	\draw [blue, ->-=.5] (0,6)--(3,0);
	\draw [blue, ->-=.5] (3,6)--(6,0);	
         \draw (0,6) -- (6,6);
         \draw (6,0) -- (6,6);

\draw[red ] (5.3,0) -- (5.3,6);
\end{tikzpicture}
\qquad \qquad 
\begin{tikzpicture}[scale =.7]
\tikzset{->-/.style={decoration={ markings,
        mark=at position #1 with {\arrow[scale=1.5,>=stealth]{>}}},postaction={decorate}}}
\node[below] at (3,0) {$\gamma_0$};
\node[above] at (2.5,2) {$\gamma_1$};
\node[left] at (1.5,1) {$\gamma_2$};
\node[above left] at (0,4) {$x_2 \cdot b$};
\node[below right] at (4,0) {$x_1 \cdot a$};
\node[below left] at (2,0) {$y$};
\draw[black, thick, ->-=.5] (2.1,0)--(4,0);
\draw[black, thick,->-=.5] (0,4)--(2-.05,0+.1);
\draw[black, thick,->-=.5] (4,0)--(0,4);
\draw[red] (.75, -.25)--(.75,4);
\draw[red] (1.75, -.25)--(1.75,4);
\draw[red] (2.75, -.25)--(2.75,4);
\draw[red] (3.75, -.25)--(3.75,4);
\end{tikzpicture}
\end{center}

\subsection{Example}
\label{three:id-map}
Let $\underline{\Lambda} = \Lambda_{\underline{C}}$ be as in \S\ref{three:Ffield} and let $\Ls{0}$ and $\phi^s \Ls{0}$ be as in \S\ref{ex-idmap}.  
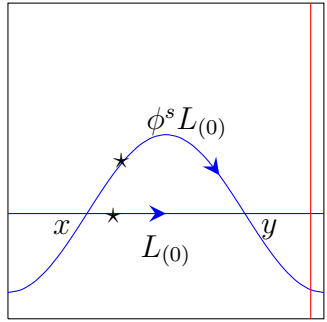
\begin{center}
\begin{tikzpicture}[scale=0.7]
      \tikzset{->-/.style={decoration={ markings,
        mark=at position #1 with {\arrow[scale=2,>=stealth]{>}}},postaction={decorate}}}
         \draw (0,-1)--(6,-1)--(6,6-1)--(0,6-1)--(0,-1);
         \draw[blue, ->-=.5]  (0,1) -- (6,1);
         \draw[domain=0:6,smooth,variable=\x,blue, ->-=.65] plot ({\x},{1-1.5*cos(60*\x)});
         \node at (2,0.97) {$\star$};
         \node at (2.15,2) {$\star$};
         \node[above left] at (1.4,0.4) {$x$};
         \node[above right] at (4.6,0.3) {$y$};
         \node[above] at (3.4,2.2) {$\phi^s \Ls{0}$};
         \node[below] at (3,0.8) {$\Ls{0}$};
         \draw[red] (5.75,-1)--(5.75,5);
\end{tikzpicture}
\end{center}
We will compute the differential on $\CF(\phi^s \Ls{0},\Ls{0};\underline{\Lambda})$ --- this specializes to the example of \S\ref{ex-idmap} in case $\sigma$ is trivial.  As in that example we still have $\CF^0(\phi^s \Ls{0},\Ls{0};\underline{\Lambda}) = x \Lambda$ and $\CF^1(\phi^s \Ls{0},\Ls{0};\underline{\Lambda}) = y\Lambda$, but the map $\mu_1$ is not $\Lambda$-linear so we must compute not just $\mu_1(x)$ but $\mu_1(x a)$ for all $a \in \Lambda$.  The same two bigons in \S\ref{ex-idmap}, of area $A$, contribute to $\mu_1(xa)$, but only of them crosses the ``danger line''
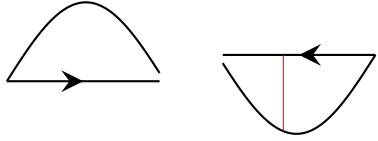
\begin{center}
\begin{tikzpicture}[scale = 0.7]
  \tikzset{->-/.style={decoration={ markings,
        mark=at position #1 with {\arrow[scale=2,>=stealth]{>}}},postaction={decorate}}}
\draw[thick, domain=1.5:4.4,smooth,variable=\x,black] plot ({\x},{.5-1.5*cos(60*\x)});
 \draw[thick, black, ->-=.5]  (1.5,.5) -- (4.4,.5);
  \draw[thick, domain=4.6:7.5,smooth,variable=\x,black] plot ({\x+1},{1-1.5*cos(60*\x)});
 \draw[thick, black, ->-=.5]  (8.5,1) -- (5.6,1);
 \draw[red] (6.75,-.45)--(6.75,1);
\end{tikzpicture}
\end{center} 
The left bigon contributes $y \cdot (-t^A a)$ and the right bigon contributes $y \cdot (t^A \sigma(a))$, so that the differential is given by
\[
\mu_1(x a) = y t^A(\sigma(a) - a)
\]
If we make the change of basis $(x,y) \to (x,yt^A)$, then 
\[
\HF^0 \cong \ker(a \mapsto \sigma(a) - a) \qquad \HF^1 \cong \coker(a \mapsto \sigma(a) -a)
\]
More suggestively, $\HF^i(\phi^s \Ls{0},\Ls{0};\underline{\Lambda})$ is isomorphic to $H^i(\Ls{0};\underline{\Lambda}\vert_{\Ls{0}})$, the cohomology of the circle $\Ls{0}$ with coefficients in $\underline{\Lambda}$.

\subsection{Example}
\label{subsec:ubiquinone}
With $\underline{\Lambda}$ and $\phi$ as in the previous example, let us compute $\CF(\phi^s L,L';\underline{\Lambda})$ when $L$ and $L'$ are two special Lagrangians that are parallel to $\Ls{0}$ and to each other, but that do not intersect.  Let $c$ be the distance between (and therefore also the area between) $L$ and $L'$.  Suppose that $|s|$ slightly exceed $c$, so that $\phi^s(L) \cap L'$ has two intersection points that we again denote by $x$ and $y$.
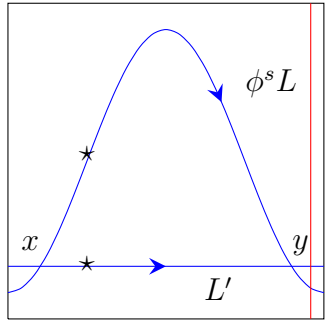
\begin{center}
\begin{tikzpicture}[scale=0.7]
      \tikzset{->-/.style={decoration={ markings,
        mark=at position #1 with {\arrow[scale=2,>=stealth]{>}}},postaction={decorate}}}
         \draw (0,-1)--(6,-1)--(6,6-1)--(0,6-1)--(0,-1);
         \draw[blue, ->-=.5]  (0,0) -- (6,0);
         \draw[domain=0:6,smooth,variable=\x,blue, ->-=.65] plot ({\x},{2-2.5*cos(60*\x)});
         \node at (1.5,0.05) {$\star$};
         \node at (1.5,2.15) {$\star$};
         \node[above left] at (0.8,.1) {$x$};
         \node[above right] at (5.2,0) {$y$};
         \node[above] at (5,3) {$\phi^s L$};
         \node[below] at (4,0) {$L'$};
         \draw[red] (5.75,-1)--(5.75,5);
\end{tikzpicture}
\end{center}
Then the differential on $\CF(\phi^s L,L';\underline{\Lambda})$ is $\mu_1(x a) = y(t^a \sigma(a) - t^{A+c} a)$.  If $c$ is not zero, the complex is acyclic, but it is interesting to note that the formal series
\begin{equation}
\label{eq:431}
x \cdot \sum_{n \in \bZ} \sigma^{-n}(a) t^{nc} \qquad a \in \underline{C}_x
\end{equation}
are killed by $\mu_1$.  Since \eqref{eq:431} has an infinite ``tail'', it does not lie in $\underline{\Lambda}_x$ and does not contribute to $\CF(\phi^s L,L';\underline{\Lambda})$.
\medskip

Since $L \cap L'$ is empty the continuation maps associated to $\phi^s$ and its reverse $\phi^{-s}$
\[
\CF(L,L';\underline{\Lambda}) \to \CF(\phi^s L,L';\underline{\Lambda})\quad \text{and} \quad \CF(\phi^s L,L';\underline{\Lambda}) \to \CF(L,L';\underline{\Lambda})
\]
 both vanish.  But the explicit contracting homotopy on $\CF^1(\phi^s L,L';\underline{\Lambda}) \to \CF^0(\phi^s L,L';\underline{\Lambda})$ is interesting, it is given by the series
\begin{equation}
\label{eq:dhap}
\Delta(a \cdot y) = x (t^{-A} \sigma^{-1}(a) +  t^{-A+c} \sigma^{-2}(a) + \cdots +  t^{-A+nc} \sigma^{-n-1}(a) + \cdots)
\end{equation}
The strip $u_n$ contributing the $t^{-A+nc}$ term in this series is the same as in \S\ref{subsec:ex-parallel}, but that contribution is now multiplied by $\nabla \gamma' \left(a \cdot \nabla (\phi^{\tau s} (y)) \nabla \gamma \nabla( \phi^{-\tau s}(x)) (1_{\underline{\Lambda}_x})\right)$, which simplifies to $\sigma^{-n-1}(a)$.

\subsection{Computing the triangle maps}
Let $\Ls{m}$ (equipped with the same orientations and stars), $x_{m,\kappa}$, and $\tau$ be as in \S\ref{two:tri-maps}.  Let $\underline{\Lambda}$ and $\underline{C}$ be pulled back along $\mathfrak{f}$, with $\sigma$ denoting the nontrivial monodromy, as in \S\ref{three:Ffield}.  We will compute 
\[
\CF(\Ls{m_1},\Ls{m_1+m_2};\underline{\Lambda}) \times \CF(\Ls{0},\Ls{m_1};\underline{\Lambda}) \to \CF(\Ls{0},\Ls{m_1+m_2};\underline{\Lambda})
\]
by computing $\mu_2(\tau^{m_1}x_{m_2,\kappa_2} \cdot b, x_{m_1,\kappa_1} \cdot a)$.

\begin{theorem*}
Let $E$ be as in \eqref{eq:323}.  Then in the absolute case $\mu_2(\tau^{m_1} x_{m_2,\kappa_2},x_{m_1,\kappa_1})$ is given by
\begin{equation}
\sum_{\ell \in \bZ} x_{m_1+m_2,E(\kappa_1,\kappa_2+\ell)} t^{(\ell+\kappa_2 - \kappa_1)^2/(2(\frac{1}{m_1} + \frac{1}{m_2}))} \sigma^{\ell-\lfloor E(\kappa_1,\kappa_2+\ell)\rfloor}(b) \sigma^{-\lfloor E(\kappa_1,\kappa_2+\ell)\rfloor}(a)
\end{equation}
where we understand $x_{m_1+m_2,E(\kappa_1,\kappa_2+\ell)} := x_{m_1+m_2,\kappa}$ if $E(\kappa_1,\kappa_2+\ell) = \kappa$ modulo $1$.
In the relative case, $\mu_2(\tau^{m_1} x_{m_2,\kappa_2},x_{m_1,\kappa_1})$ is given by
\begin{equation}
\label{eq:monomial-product2}
\sum_{\ell \in \bZ} x_{m_1+m_2,E(\kappa_1,\kappa_2+\ell)} t^{\lambda(\kappa_1,\kappa_2+\ell)} \sigma^{\ell-\lfloor E(\kappa_1,\kappa_2+\ell)\rfloor}(b) \sigma^{-\lfloor E(\kappa_1,\kappa_2+\ell)\rfloor}(a)
\end{equation}
\end{theorem*}

\begin{proof}
The triangles that contribute are exactly as in the proof in \S\ref{two:tri-maps}, we will index them again by $\ell \in \bZ$.  The $\pm$ sign and the exponent of $t$ in \eqref{eq:175} are the same as in \S\ref{two:tri-maps}, but it remains to compute $\nabla \gamma_2(b \nabla \gamma_1(a \nabla \gamma_0(1)))$.  If $I$ is an interval and $\gamma:I \to T$ is a path in $T$, write $\mathfrak{f}(\gamma)$ for the number of times $\gamma$ crosses the ``danger line'' \S\ref{three:Ffield}, counted with sign.  Then 
\begin{eqnarray}
\nabla \gamma_2(b \nabla \gamma_1(a \nabla \gamma_0(1))) & = & \sigma^{\mathfrak{f}(\gamma_2)}(b \sigma^{\mathfrak{f}(\gamma_1)}(a)) \\
& = & \sigma^{\mathfrak{f}(\gamma_2)}(b) \sigma^{\mathfrak{f}(\gamma_1)+\mathfrak{f}(\gamma_2)}(a)\\
& = & \sigma^{\mathfrak{f}(\gamma_2)}(b) \sigma^{-\mathfrak{f}(\gamma_0)}(a) \label{eq:364p}
\end{eqnarray}
The $\ell$th triangle (pictured below) has $\mathfrak{f}(\gamma_2) = \lfloor E(\kappa_1,\kappa_2+\ell)\rfloor$ and $\mathfrak{f}(\gamma_0) = \ell - \lfloor E(\kappa_1,\kappa_2+\ell) \rfloor$. 

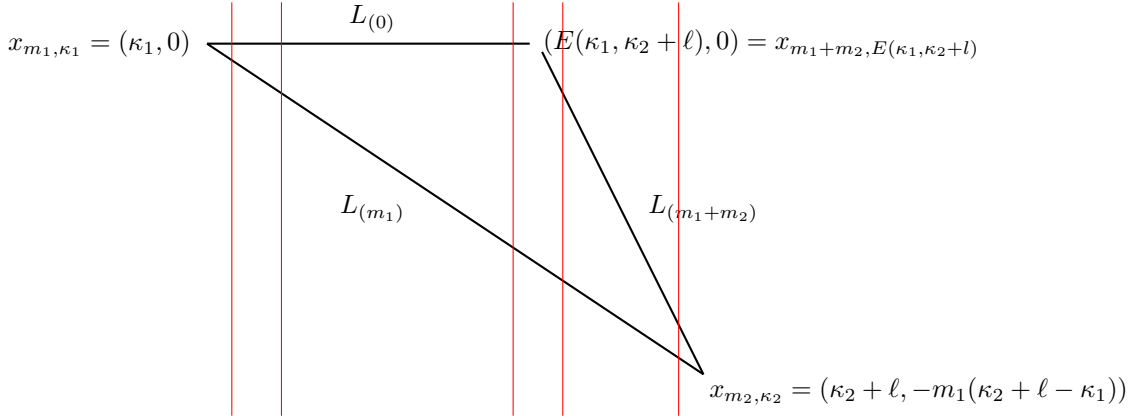
\begin{center}
\begin{tikzpicture} [scale=1.1]
      \tikzset{->-/.style={decoration={ markings,
        mark=at position #1 with {\arrow[scale=2,>=stealth]{>}}},postaction={decorate}}}
        
        \node at (-2,6.3) {\footnotesize $\Ls{0}$};
        \node at (2,4) {\footnotesize $\Ls{m_1+m_2}$};
        \node at (-2,4) {\footnotesize $\Ls{m_1}$};

        \draw[thick, black] (-4,6) -- (-.1,6);
	\node at (-5.3,6) {\footnotesize $x_{m_1,\kappa_1}= (\kappa_1, 0)$};
        \draw[thick, black] (0+.05,6-.1) -- (2,2);
         \node at (4.6,1.8) {\footnotesize 
	 $x_{m_2,\kappa_2}= \left( \kappa_2+\ell , -m_1(\kappa_2+\ell-\kappa_1) \right)$ };     
        \draw[thick, black] (-4,6) -- (2,2);
	\node at (2.7,6) {\footnotesize $(E(\kappa_1,\kappa_2+\ell), 0) = x_{m_1+m_2, E(\kappa_1,\kappa_2+l)}$};
	\draw[red ] (-3.7,1.5) -- (-3.7,6.5);	
	\draw[red ] (-3.1,1.5) -- (-3.1,6.5);

	\draw[red ] (1.7,1.5) -- (1.7,6.5);
	\draw[red ] (-0.3,1.5) -- (-0.3,6.5);
	\draw[red ] (0.3,1.5) -- (0.3,6.5);

\end{tikzpicture}
\end{center}

\end{proof}

\subsection{Theta functions with $F$-field coupling} 
\label{three:last}
Keeping the notation of the previous section, where $\underline{\Lambda}$ and $\underline{C}$ are pulled back along $\mathfrak{f}$, we may give
\begin{equation}
\label{eq:181}
\bigoplus_{m = 1}^{\infty} \CF(\Ls{0},\Ls{m};\underline{\Lambda})
\end{equation}
the structure of a graded ring without unit.  One may equip it with a unit by taking the degree zero piece to be $\Lambda_{C^{\sigma}} = \HF^0(\Ls{0},\Ls{0};\underline{\Lambda})$ \S\ref{three:id-map}.  Here is a description of \eqref{eq:181} analogous to that of \S\ref{two:theta}:

\begin{theorem*}[\S\ref{intro:thm1}]
For each $a \in C$ and each pair of integers $m,k$ with $m > k \geq 0$, let $\theta_{m,k}[a]$ denote the formal series
\begin{equation}
\label{eq:thetamka}
\theta_{m,k}[a] := \sum_{i = -\infty}^{\infty} t^{m\frac{i(i-1)}{2} + ki} z^{mi + k} \sigma^i(a)
\end{equation}
Let $\tabs_{m,k}[a]$ denote the formal series
\begin{equation}
\label{eq:thetamka_abs}
\tabs_{m,k}[a] := \sum_{i= -\infty}^{\infty} t^{\frac{1}{2m}(mi + k)^2} z^{mi+ k} \sigma^i(a)
\end{equation}
Then the relative (resp. absolute) version of \eqref{eq:181} is isomorphic (as a graded ring-without-unit) to the $\bZ[\![t]\!]$-linear span of the $\theta_{m,k}[a]$ (resp. $\Lambda_{\bZ}$-linear span of the $\tabs_{m,k}[a]$) via the map
\begin{equation} 
\label{eq:ABmap}
x_{m,k/m} \cdot a \mapsto \theta_{m,k} [a] 
\end{equation}
\end{theorem*}

As in \S\ref{two:theta}, the summands of the $\theta[a]$ are indexed by right triangles.  The factor of $\sigma^i(a)$ plays the same role as the $\nabla \gamma_2(b \nabla \gamma_1(a \nabla(\gamma_0(1)))$ factor in \eqref{eq:175}.
\begin{center}
\begin{tikzpicture}
\tikzset{->-/.style={decoration={ markings,
        mark=at position #1 with {\arrow[scale=1.5,>=stealth]{>}}},postaction={decorate}}}
\draw[black, thick, ->-=.5] (0.1,0)--(1,0);
\draw[black,thick,->-=.5] (1,0)--(0,2);
\draw[black,thick,->-=.5] (0,2)--(0,.1);
\node[below right] at (1,0) {$a \cdot x_{2,1}$};
\node[above right] at (.5,1) {$\sigma^2$};
\node[left] at (0,.75) {$z^{2\cdot2+ 1} = z^{5}$};
\draw[red] (0-.1,-.2)--(0-.1,2.1);
\draw[red] (.4-.1,-.2)--(.4-.1,2.1);
\draw[red] (.8-.1,-.2)--(.8-.1,2.1);
\draw[red] (1.2-.1,-.2)--(1.2-.1,2.1);
\end{tikzpicture}
\end{center}
This triangle contributes $t^9 z^5 \sigma^5(a)$ to the relative version of $\theta_{2,1}[a]$ \eqref{eq:thetamka}, and $t^{6.25} z^5 \sigma^5(a)$ to the absolute version \eqref{eq:thetamka_abs}.
Presumably a ``family Floer'' argument along these lines would prove the Theorem, but we will give a proof in terms of the explicit formulas.

\begin{proof}
Let us give the proof first in the relative case.  Fix $m_1,m_2 \in \bZ_{\geq 0}$, $k_1 \in \{0,\ldots,m_1 - 1\}$, $k_2 \in \{0,\ldots,m_2 - 1\}$, and $a,b \in C[\![t]\!]$.  The product of $\theta_{m_2,k_2}[b]$ and $\theta_{m_1,k_1}[a]$ is by definition
\begin{equation}
\label{eq:i1i2-sum}
\sum_{i_1,i_2 \in \bZ \times \bZ} \sigma^{i_2}(b)\sigma^{i_1}(a) t^{m_2\binom{i_2}{2}  + m_1 \binom{i_1}{2}  + k_2i_2 + k_1i_1} z^{m_2 i_2 + m_1 i_1 +  k_2 + k_1}
\end{equation}
We may also index the sum by triples $(r,c,d)$ where $(c,d) \in \bZ \times \bZ$ and $r \in \left\{0,1,\ldots, \frac{m_1 + m_2}{\mathrm{gcd}(m_1,m_2)}-1\right\}$.  First, for $\ell \in \mathbb{Z}$, we define $d$ and $r$ via 
\[ \ell = \frac{m_1+m_2}{\mathrm{gcd}(m_1,m_2)}d + r \]
It then follows that if we set
\[
e(r) = e_{m_1,m_2,k_1,k_2}(r) :=  \left\lfloor \frac{m_2 r + k_1 + k_2}{m_1 + m_2} \right\rfloor \in \left\{0,1,\ldots, \frac{m_2}{\mathrm{gcd}(m_1,m_2)}-1\right\} 
\]
and $\kappa_1 = k_1/m_1, \kappa_2 = k_2/m_2$, we have 
\begin{equation} \label{eq:eqE} \lfloor E(\kappa_1,\kappa_2+\ell) \rfloor = \frac{m_2 d}{\mathrm{gcd}(m_1,m_2)} + e(r)
\end{equation}
The triple $(r,c,d)$ (and the integer $\ell$) is determined as the unique solution to 
\[ \label{eq:reindex}
\left(
\begin{array}{r}
i_1 \\ i_2
\end{array}
\right) = \left(
\begin{array}{rr}
1 & -m_2/\mathrm{gcd}(m_1,m_2) \\ 
1 & m_1/\mathrm{gcd}(m_1,m_2)
\end{array}
\right)
\left(
\begin{array}{c}
c -e(r) \\ d
\end{array}
\right)+ \left(
\begin{array}{r}
0 \\ r
\end{array}
\right)
\]
To save space in the exponents, let us write $g := \gcd(m_1,m_2)$.  After reindexing the sum \eqref{eq:i1i2-sum} is
\begin{equation}
\label{eq:rcd-sum}
\sum_r \sum_c \sum_d \sigma^{c-e(r) + m_1d/g + r}(b)  \sigma^{c - e(r) - dm_2/g}(a)  t^{\square_1}z^{\square_2}
\end{equation}
where
\begin{equation}
\label{eq:square1}
\begin{array}{rcc}
\square_1 & = & m_1\binom{c-e(r) - dm_2/g}{2}+ k_1(c-e(r) - dm_2/g) \\
& & + m_2\binom{c-e(r) + dm_1/g + r}{2} + k_2(c-e(r) + dm_1/g + r)
\end{array}
\end{equation}
and
\begin{equation}
\label{eq:square2}
\square_2 = m_1(c-e(r) - dm_2/g) + m_2(c-e(r) + dm_1/g + r) + k_1 + k_2
\end{equation}

We note that $\square_2 = (m_1 + m_2)c + (m_2 r + k_1 + k_2) - (m_1+m_2)e(r)$ does not depend on $d$, and furthermore that 
\[
k(r) = k_{m_1,m_2,k_1,k_2}(r) := m_2 r + k_1 + k_2 - (m_1+m_2)e(r)
\]
belongs to $\{0,\ldots,m_1+m_2-1\}$. (Note that $k(r)/(m_1+m_2)$ is the fractional part of $E(\kappa_1,\kappa_2+\ell)$.)
Thus the sum \eqref{eq:rcd-sum} is the same as
\[
\sum_r \sum_c \left(\sum_d \sigma^{c -e(r)+ dm_1/g + r}(b)  \sigma^{c -e(r) - dm_2/g}(a)   t^{\square_1} \right)z^{(m_1 + m_2)c + k(r)}
\]
Since $\sigma$ acts trivially on $t$, this is the same as 
\[
\sum_r \sum_c \sigma^c\left(\sum_d \sigma^{-e(r)+dm_1/g + r}(b)   \sigma^{-e(r) - dm_2/g}(a) t^{\square_1} \right)z^{(m_1 + m_2)c + k(r)}
\]
Now we claim
\begin{equation}
\label{eq:claim}
\square_1 =
\lambda\left(\kappa_1,\kappa_2+\ell \right) + (m_1 + m_2) \binom{c}{2} + k(r) c
\end{equation}
where $\square_1$ is as in \eqref{eq:square1} and $\lambda$ is as in \eqref{eq:Euv-lamuv} and $\ell = \frac{m_1+m_2}{g}d + r$.

Taking \eqref{eq:claim} for granted, we obtain that $\theta_{m_2,k_2}[b]\cdot \theta_{m_1,k_1}[a]$ is equal to
\[
\sum_{r = 0}^{\frac{m_1+m_2}{g}-1} \sum_c \sigma^c\left(\sum_d  \sigma^{-e(r) - dm_2/g}(a) \sigma^{-e(r)+dm_1/g + r}(b) t^{\lambda(\kappa_1, \kappa_2+\ell)} \right)t^{(m_1 + m_2)\binom{c}{2} + k(r) c}z^{(m_1 + m_2)c + k(r)}
\]
which is equal to
\begin{equation}
\label{eq:after-claim}
\sum_r \theta_{m_1+m_2,k(r)}\left[\sum_d \sigma^{-e(r)+dm_1/g + r}(b)  \sigma^{-e(r) - dm_2/g}(a)    t^{\lambda(\kappa_1, \kappa_2+\ell)}\right]
\end{equation}

We now compare \eqref{eq:after-claim} to $(b \cdot x_{m_2,k_2/m_2}) (a \cdot x_{m_1,k_1/m_1})$ which is given by \eqref{eq:monomial-product2}
\[
\sum_{\ell \in \bZ} \left(\sigma^{\ell - \lfloor E(\kappa_1,\kappa_2 + \ell) \rfloor}\!(b) \, \sigma^{-\lfloor E(\kappa_1,\kappa_2 + \ell)\rfloor}\!(a) \,  t^{\lambda(\kappa_1,\kappa_2 + \ell)}\right) x_{m_1+ m_2,E(\kappa_1,\kappa_2 + \ell)}
\]
Writing $\ell = \frac{m_1+m_2}{\mathrm{gcd}(m_1,m_2)} d + r$ and using the formula \eqref{eq:eqE}, we can rewrite this as
\begin{equation}\label{eq:Aside} \sum_r  
\left(\sum_d  \sigma^{-e(r)+dm_1/g + r}(b)  \sigma^{-e(r) - dm_2/g}(a)  t^{\lambda(\kappa_1, \kappa_2+ \ell)}  \right)  x_{m_1+ m_2,k(r)}
\end{equation}
It is now evident that under our map \eqref{eq:ABmap}, the two expressions \eqref{eq:after-claim} and \eqref{eq:Aside} agree.

It remains to verify the claim in \eqref{eq:claim}. This will be a direct computation. We have
\begin{align*} &\lambda(\kappa_1,\kappa_2+ \ell) = m_1 \phi(\kappa_1) + m_2 \phi(\kappa_2 + \ell) - (m_1+m_2) \phi( \frac{m_2d}{g}+e(r) + \frac{k(r)}{m_1+m_2})  \\
&= m_2 ( \ell(\ell+\kappa_2) - \frac{\ell(\ell+1)}{2} ) \\
& - (m_1+m_2) \left((\frac{dm_2}{g}+e(r))(\frac{dm_2}{g}+e(r) + \frac{k(r)}{m_1+m_2}) - \frac{(\frac{dm_2}{g}+e(r))(\frac{dm_2 }{g}+e(r)+1)}{2}) \right) \\
&= m_2 ( \ell \kappa_2 + \frac{\ell(\ell-1)}{2}) - (\frac{dm_2 }{g}+e(r)) k(r) - (m_1+m_2) \frac{(\frac{dm_2 }{g}+e(r)) (\frac{dm_2 }{g}+e(r)-1)}{2}  \\
&= m_2 \kappa_2 (\frac{d(m_1+m_2)}{g} +r) + \frac{m_2 (\frac{d(m_1+m_2)}{g}+r) ( \frac{d(m_1+m_2)}{g}+r-1)}{2} \\ &- (\frac{dm_2 }{g}+e(r)) k(r) - (m_1+m_2) \frac{(\frac{d m_2 }{g}+e(r)) (\frac{d m_2 }{g}+e(r)-1)}{2}
\end{align*}
Substituting in $k(r) = m_2 r + k_1 + k_2 - (m_1+m_2)e(r)$, we get
\begin{align*} \lambda(\kappa_1,\kappa_2+\ell) &= m_2 \kappa_2 (\frac{d(m_1+m_2)}{g} +r) + \frac{m_2 (\frac{d(m_1+m_2)}{g}+r) ( \frac{d(m_1+m_2)}{g}+r-1)}{2} \\ &- (\frac{dm_2}{g}+e(r)) (m_1\kappa_1 + m_2 \kappa_2 + m_2r - e(r)(m_1+m_2)) \\ &- (m_1+m_2) \frac{ (\frac{dm_2 }{g}+e(r)) (\frac{dm_2}{g}+e(r)-1)}{2}
\\ &= \frac{(m_1+m_2) m_1 m_2}{2 g^2} d^2 + \frac{m_1 m_2}{g}  (r+ \kappa_2 -\kappa_1) d \\ &+ m_2 \kappa_2 r + \frac{m_2 r (r-1)}{2} + \frac{(m_1+m_2) e(r)(e(r)+1)}{2} - (m_1\kappa_1 +m_2 \kappa_2 + m_2 r)e(r)  \end{align*}
We can rewrite this in a symmetric form as follows:
\begin{align*} \lambda(\kappa_1,\kappa_2+\ell) &= \frac{(m_1+m_2) m_1 m_2}{2 g^2} d^2 + \frac{m_1 m_2}{g}  (r+ \kappa_2 -\kappa_1) d \\ &+ \frac{m_2(e(r)-r)(e(r)-r+1)}{2} + \frac{m_1e(r)(e(r)+1)}{2} - k_2(e(r)-r) - k_1e(r)
\end{align*}
and this in turn can be seen to be equal to 
\begin{align*}
 \lambda(\kappa_1,\kappa_2+\ell) &=  m_1\binom{-e(r) - m_2d/g}{2}+ k_1(-e(r) - m_2d/g) \\ &+ m_2\binom{-e(r) + m_1d/g + r}{2} + k_2(-e(r) + m_1d/g + r)
\end{align*}
This completes the proof of the claim \eqref{eq:claim} and hence the proof of the theorem in the relative case.

In the absolute case, the product of $\tabs_{m_2,k_2}[b]$ and $\tabs_{m_1,k_1}[a]$ is given by 
\begin{equation}
\label{eq:i1i2-abssum}
\sum_{i_1,i_2 \in \bZ \times \bZ} \sigma^{i_2}(b)\sigma^{i_1}(a) t^{\frac{(m_2 i_2 + k_2)^2}{2m_2} + \frac{(m_1i_1 +k_1)^2}{2m_1}} z^{m_2 i_2 + m_1 i_1 +  k_2 + k_1}
\end{equation}
Performing the same re-indexing using \eqref{eq:reindex}, we arrive at  
\begin{equation}
\label{eq:rcd-sumabs}
\sum_r \sum_c \sum_d \sigma^{c-e(r) + m_1d/g + r}(b)  \sigma^{c - e(r) - dm_2/g}(a)  t^{\square^{abs}_1}z^{\square^{abs}_2}
\end{equation}
where
\begin{equation}
\label{eq:square1abs}
\square^{abs}_1 = \frac{(m_1(c-e(r) - dm_2/g) + k_1)^2}{2m_1} + \frac{(m_2(c-e(r) + dm_1/g + r) + k_2)^2}{2m_2} 
\end{equation}
and $\square_2^{abs}=\square_2$ given as before by \eqref{eq:square2}. 
Following the same steps, the only difference in the calculation is the verification of the analogue of equation \eqref{eq:claim} which now takes the form:
\begin{equation}
\label{eq:claimabs}
\square^{abs}_1 =
\frac{(\ell+\kappa_2-\kappa_1)^2 m_1 m_2}{2(m_1+m_2)}  + \frac{((m_1 + m_2)c + k(r))^2}{2(m_1+m_2)}
\end{equation}
Recalling that $\ell = (m_1+m_2)d/g + r$,  $k(r) = m_2 r + k_1 + k_2 - (m_1+m_2)e(r)$ and $\kappa_i = k_i/m_i$ for $i=1,2$, we can compare the equations \eqref{eq:square1abs} and \eqref{eq:claimabs} directly to verify the claim. This completes the proof in the absolute case.
\end{proof}

\section{Specializing the Novikov parameter}

\subsection{At $t = 0$}
\label{subsec:t=0}  The specialization $t = 0$ renders uninteresting the absolute version of the maps $\mu_n$, at least if we also set $t^a = 0$ for every $a>0$.  But it is a standard part of relative Floer theory.  In fact it is part of the motivation for relative Floer theory --- in any sum over triangles (say), the contribution from triangles which are not disjoint from $D$ vanishes, so that working with $t = 0$ is closely related to replacing the closed symplectic manifold $T$ with the open $T -D$.  See \cite[\S 6.1]{LP} for some more context.

For short, let us write $S_n$ for the $n$th graded piece of \eqref{eq:181}.  Let us also put $S_0 := C^{\sigma}[\![t]\!]$ --- here $C^{\sigma}$ denotes the $\sigma$-fixed subring of $C$.  Then $S_{\bullet}$ is a graded $C^{\sigma}[\![t]\!]$-algebra --- it is associative and commutative by Theorem \ref{intro:thm1}, \S\ref{three:last}.  If $C$ is a perfect field and $\sigma$ is the $p$th root map, then $C^{\sigma} = \bF_p$.  In any case there is an isomorphism in the category of $C^{\sigma}$-schemes
\[
\Proj(S \times_{C^{\sigma}[\![t]\!]} C^{\sigma}) = \mathrm{colim}\left[
\begin{tikzcd}
\Spec(C) \ar[r,shift left=.75ex,"i_0 \circ \sigma"]
  \ar[r,shift right=.75ex,swap,"i_{\infty}"]
&
\bP^1_{/C}
\end{tikzcd}
\right]
\]
where $i_0$ and $i_{\infty}$ are the inclusions of $C$-schemes $\Spec(C) \to \bP^1_{/C}$ with coordinates $0$ and $\infty$, respectively.

If $C$ is a field then $\Proj(S \times_{S_0} C^{\sigma})$ is a one-dimensional scheme, which can be covered by two affine charts.  It fails to be regular at a unique point and the complement of this point is isomorphic to $\Spec(C[x,x^{-1}])$.  For the other chart take the complement of any other point --- one obtains an affine Zariski neighborhood of the non-regular point that is isomorphic to the spectrum of a subring of $C[y]$, namely
\[
\{f \in C[y] : \sigma(f(0)) = f(1)\}
\]
This ring is in some sense an order in a Dedekind domain but  if $C$ has infinite degree over $C^{\sigma}$ then it is not of finite type.

\subsection{Floer cochains at $t = 1$}
Let $C$ and $\sigma$ be as in \S\ref{three:Ffield}, with $\underline{C}$ pulled back along $\mathfrak{f}$ from the sheaf whose \'etal\'e space is \eqref{eq:etale-space}.  If $L$ and $L'$ are one-dimensional submanifolds that intersect transversely, we will write (similar to \eqref{eq:CFLLp})
\begin{equation}
\label{eq:501}
\CF(L,L';\underline{C}) = \bigoplus_{x \in L \cap L'} \underline{C}_x
\end{equation}
This supports a $\bZ/2$-grading and a bigon differential
\begin{equation}
\label{eq:501p}
\mu_1(x \cdot a) = \sum_{y | \mas(y) = \mas(x)+1} y \left(\sum_{u \in \cM(y,x)} \pm \nabla \gamma'\left(a \nabla \gamma(1_{\underline{\Lambda_y}})\right) \right)
\end{equation}
with $\gamma$ and $\gamma'$ as in \eqref{eq:423}.  \eqref{eq:501p} is a finite sum

If we further endow $C$ with a topology, for which $\sigma$ is continuous, we can investigate the algebraic structures on \eqref{eq:501} induced by \eqref{eq:441}.  That is, we study the sums 
\begin{equation}
\label{eq:502}
\sum_y y \left(\sum_{u} \pm \nabla \gamma_n(a_n \nabla \gamma_{n-1}(a_{n-1} \cdots \nabla \gamma_2(a_2 \nabla(\gamma_1(a_1 \nabla \gamma_0(1))\cdots)))) \right)
\end{equation}
\eqref{eq:501p} and \eqref{eq:502} are simply the specializations one obtains by setting $t$ and every power $t^a$ to $1$ in the formulas from \S\ref{three}.
The sums $\sum_u$ in \eqref{eq:502} might diverge or converge in the topological ring $C$, so that at best the map
\begin{equation}
\label{eq:503}
\CF(L_{n-1},L_n;\underline{C}) \times \cdots \times \CF(L_0,L_1;\underline{C}) \dashrightarrow \CF(L_0,L_n;\underline{C})
\end{equation}
is only partially defined.  In many cases, the domain of convergence is reduced to a point, but we will see that the triangle maps are not trivial.

\subsection{The triangle products at $t = 1$}
Suppose that $C$ is complete with respect to a nonarchimedean norm $|\cdot|$, and that
\[
\sigma(c) = |c|^{1/p}
\]
for some $p > 1$.  With $p$ prime and $C,\sigma$ as in \eqref{eq:192},  the pair $(C,|\cdot|)$ is a perfectoid field of characteristic $p$ \cite[\S 3]{Scholze}.  The maps $\nabla \gamma$ for the sheaf of rings $\underline{C}$ are continuous but (crucially) they d not preserve the norms.

Write 
\[
\cO_C := \{c \in C : |c|\leq 1\} \qquad \mathfrak{m}_C:=\{c \in C:|c|<1\};
\]
then $\cO_C$ is the ring of integers in $C$ and $\mathfrak{m}$ is the unique maximal ideal of $\cO_C$.  They are both stable by the $\sigma$-action so that they determine locally constant subsheaves of $\underline{C}$ that we denote by $\underline{\cO}_C$ and $\underline{\mathfrak{m}}_C$.  The fiber of $\underline{\mathfrak{m}}$ at $x$ is the set of topologically nilpotent elements in $\underline{C}_x$.

Following the notation of \eqref{eq:501} set
\begin{equation}
\CF(L,L';\underline{\mathfrak{m}}):=\bigoplus_{x \in L \cap L'} \underline{\mathfrak{m}}_x
\end{equation}
It is an open subgroup of $\CF(L,L';\underline{C})$.

Suppose $L_0$, $L_1$, and $L_2$ are special of finite slopes $m_0$, $m_1$, and $m_2$, in the sense of \S\ref{ex:sl}.  
The contribution of a triangle with vertices
\[
x_1 \in L_0 \cap L_1 \qquad x_2 \in L_1 \cap L_2 \qquad y \in L_2 \cap L_0
\]
to $\mu_2(x_2 \cdot b, x_1 \cdot a)$ has the form \eqref{eq:364p} $\sigma^{\mathfrak{f}(\gamma_2)}(b) \sigma^{-\mathfrak{f}(\gamma_0)}(a)$, where $\gamma_0$ and $\gamma_2$ are the two edges of $u$ incident with the output vertex $y$.  If (and only if) $m_0 < m_1 < m_2$, then $\mathfrak{f}(\gamma_2)$ and $\mathfrak{f}(\gamma_0)$ all have the same sign --- with perhaps finitely many exceptions where one of $\mathfrak{f}(\gamma_0)$ and $\mathfrak{f}(\gamma_2)$ are zero --- so that when $|a|< 1$ and $|b|<1$
\[
|\sigma^{\mathfrak{f}(\gamma_2)}(b) \sigma^{-\mathfrak{f}(\gamma_0)}(a)| = |b|^{p^{-\mathfrak{f}(\gamma_2)}}|a|^{p^{\mathfrak{f}(\gamma_0)}}
\]
is very rapidly decreasing as the side lengths of the triangles go to infinity.  The triangle product
\[
\mu_2:\CF(L_1,L_2;\underline{\mathfrak{m}}) \times \CF(L_0,L_1;\underline{\mathfrak{m}}) \to \CF(L_0,L_2;\underline{\mathfrak{m}})
\]
is therefore convergent when $m_0 < m_1 < m_2$.  In particular we have a graded ring (for now, without unit)
\begin{equation}
\label{eq:432p}
\bigoplus_{m = 1}^{\infty} \CF(\Ls{0},\Ls{m};\underline{\mathfrak{m}})
\end{equation}

\subsection{The irrelevant ideal in the Fargues-Fontaine graded ring}
Let $C$ be an algebraically closed field of characteristic $p$ that is complete with respect to a norm $|\cdot|$.  Let $B$ and $\varphi$ be as in \S\ref{intro:FFgr}, i.e.
\begin{equation}
\label{eq:promoter}
B = \left\{\sum_{i \in \bZ} b_i z^i \mid \forall r \in (0,1), |b_i|r^i \to 0\text{ as }|i| \to \infty\right\}
\end{equation}
This appears in \cite[Def. 21]{KS} and in \cite[Ex. 1.6.5]{FF}.  Fargues and Fontaine define a version of $B$ for every local field $E$, \eqref{eq:promoter} is the case when $E = \bF_p(\!(z)\!)$.  Below, we are taking advantage of the fact that when $E$ has equal characteristic each element of $B$ has a unique series expansion, something that is not clear when $E$ has mixed characteristic \cite[Rem. 1.6.7]{FF}.

Let $\varphi$ be as in \eqref{eq:phiBB}, i.e. the automorphism of $B$ given by $\varphi(\sum c_i z^i) = \sum c_i^p z^i$.  The homogeneous coordinate ring of $\FF_E(C)$ is
\begin{equation}
\label{eq:dune}
\bF_p(\!(z)\!) \oplus B^{\varphi = z} \oplus B^{\varphi = z^2}  \oplus \cdots
\end{equation}
We will prove the theorem of \S\ref{intro:FFgr}, i.e. that \eqref{eq:432p} is isomorphic to the irrelevant ideal of this ring.

\begin{proof}[Proof of \eqref{eq:feste}]
Suppose that $a \in C$ has $|a| < 1$.  Then the sequence $|\sigma^i(a)| = |a|^{p^{-i}}$ of real numbers is bounded as $i \to \infty$ and very rapidly decreasing as $i \to -\infty$, and $r^{mi+k} |a|^{p^{-i}} \to 0$ as $|i| \to \infty$ for any $r \in (0,1)$.  Therefore for any $m$ and $k$ the expression
\begin{equation}
\label{eq:roux}
\sum_{i \in \bZ} z^{mi+k} \sigma^i(a)
\end{equation}
belongs to $B$ \eqref{eq:FFKS}.  Applying $\varphi$ to \eqref{eq:roux} gives $\sum_{i \in \bZ} z^{mi+k} \sigma^{i-1}(a)$ --- re-indexing this series gives
\[
\sum_{i \in \bZ} z^{m(i+1) + k} \sigma^i(a) = z^m \sum_{i \in \bZ} z^{mi + k} \sigma^i(a)
\]
so that \eqref{eq:roux} belongs to $B^{\varphi = z^m}$.  But \eqref{eq:roux} is $\theta_{m,k}[a]\vert_{t = 1}$, so that by \S\ref{three:last} the map
\begin{equation}
\label{eq:boveri}
x_{m,k/m} \cdot a \mapsto \theta_{m,k}[a]\vert_{t = 1}
\end{equation}
intertwines $\mu_2$ with the ring structure on $B$.

If $f = \sum b_i z^i$ belongs to $B^{\varphi = z^m}$, then $b_{mi+k} = \sigma^m(b_k)$, so $f$ is determined by $b_0,\ldots,b_{m-1}$.  To obey \eqref{eq:FFKS}, the elements $b_0,\ldots,b_{m-1}$ must all belong to $\mathfrak{m}$.  The map
\[
f \mapsto \sum_{k = 0}^{m-1} x_{m,k/m} \cdot b_k
\]
gives the inverse isomorphism to $\CF(\Ls{0},\Ls{m};\underline{\mathfrak{m}}) \cong B^{\varphi = z^m}$.
\end{proof}

\subsection{SYZ duality}
\label{subsec:syz}
The degree one part of \eqref{eq:feste} is an isomorphism
\begin{equation}
\label{eq:hart}
\CF(\Ls{0},\Ls{1};\underline{\mathfrak{m}}) \cong \Hom_{\FF}(\cO,\cO(1))
\end{equation}
where $\cO(1)$ is the Serre line bundle on \eqref{eq:FF-curve}.  In general it seems that $\CF(L,L';\underline{\mathfrak{m}})$ captures the set of homomorphisms between two vector bundles on $\FF$ whenever $L$ and $L'$ are (or are just isotopic to, if we replace $\CF$ by $\HF$) special Lagrangians \S\ref{ex:sl} of finite slopes $m$ and $m'$ with $m$ strictly less than $m'$.  But for other kinds of homomorphisms or Ext groups in $\Coh(\FF)$, another construction must be necessary --- one that we only partially understand.  In the next two sections \S\ref{subsec:sky} and \S\ref{subsec:ore} we illustrate this in terms of skyscraper sheaves on $\FF$.

The closed points of $\FF_E(C)$ are naturally parametrized by the $\bZ$-orbits of $E$-untilts of the perfectoid field $C$.  When $E = \bF_p(\!(z)\!)$, an ``$E$-untilt'' is just a continuous homomorphism $i:E \to C$ ---  such a homomorphism must carry $z$ to a nonzero element of $\mathfrak{m}$ and conversely every nonzero element of $\mathfrak{m}$ extends to a map from $\bF_p(\!(z)\!)$, the $\bZ$-action is generated by $i \mapsto \sigma \circ i$.  There is a map
\begin{equation}
\label{eq:SYZ}
\left(\text{closed points of }\FF(E,C)\right) \to \bR/\bZ
\end{equation}
It is defined for any $E$.  When $E = \bF_p(\!(z)\!)$ it carries the $\bZ$-orbit of $\iota:\bF_p(\!(z)\!) \to C$ to the $\bZ$-coset of $\log_p(\log(|i(z)|^{-1}))$.  We expect that \eqref{eq:SYZ} is the SYZ dual to \eqref{eq:191}, and that the skyscraper sheaves have something to do with fibers of \eqref{eq:191}.

\subsection{Skyscraper sheaves and $\Ls{\infty}$}
\label{subsec:sky}
If $\zeta \in C$ is invertible, let us denote by $\Ls{\infty}^{\zeta}$ the special Lagrangian $\Ls{\infty}$ equipped with the rank one local system of $\underline{C}\vert_{\Ls{\infty}}$-modules (i.e., a local system of $C$-modules) whose fiber at $(0,0)$ is $\underline{C}_{(0,0)} = C$ and whose monodromy (in the direction of the default orientation, top to bottom) is multiplication by $\zeta$.  Let $e_m$ denote $(0,0)$ regarded as the unique intersection point of $\Ls{m}$ and $\Ls{\infty}$, so that 
\begin{equation}
\label{eq:pasteur}
\CF(\Ls{m},\Ls{\infty}^{\zeta};\underline{C}) = \CF^0(\Ls{m},\Ls{\infty}^{\zeta};\underline{C}) =  e_m \cdot C.
\end{equation}

If $C$ is algebraically closed then one also has $\Hom_{\FF}(\cO(m),\delta) \cong C$ for any skyscraper sheaf $\delta$.  If $|\zeta| < 1$ and $\delta$ is the skyscraper sheaf supported at the $\bZ$-orbit of the untilt $\bF_p(\!(z)\!) \to C$, then we expect that for any surjection $q:\cO(1) \to \delta$, there is an isomorphism making the diagram
\begin{equation}
\label{eq:only-square}
\begin{aligned}
\xymatrix{
\CF(\Ls{0},\Ls{1};\underline{\mathfrak{m}}) \ar[d]_{\eqref{eq:hart}} \ar[rr]^{\mu_2(e_1 \cdot 1,-)} & & \CF(\Ls{0},\Ls{\infty}^\zeta;\underline{C}) \ar@{-->}[d] \\
\Hom(\cO,\cO(1)) \ar[rr]_{q \circ} & & \Hom(\cO,\delta)
}
\end{aligned}
\end{equation}
commute; instead of constructing this isomorphism here let us verify that the two rows of \eqref{eq:only-square} have the same kernel.  We may find $i:\cO \to \cO(1)$ such that 
\[
0 \to \cO \xrightarrow{i} \cO(1) \xrightarrow{q} \delta \to 0
\]
is exact, so that the kernel of the bottom row in \eqref{eq:only-square} is isomorphic to $\Hom(\cO,\cO) = \bF_p(\!(z)\!)$, the ground field of \eqref{eq:FF-curve}.  We will show that the kernel of $\mu_2(e_1 \cdot 1,-)$ has the structure of a one-dimensional $\bF_p(\!(z)\!)$-module.

In general the triangle map $\mu_2(e_1 \cdot b, x_{1,0} \cdot a)$ \eqref{eq:503} is given by
\begin{equation}
\label{eq:melanogaster}
\begin{aligned}
\begin{tikzpicture}
\tikzset{->-/.style={decoration={ markings,
        mark=at position #1 with {\arrow[scale=1.5,>=stealth]{>}}},postaction={decorate}}}
\node at (-4,1) {$e_0 \cdot \left(\sum_{i \in \bZ} (-1)^{3i} b \zeta^{i} \sigma^i(a)\right)$};
\draw[black, thick, ->-=.5] (.1,0)--(2,0);
\draw[black,thick,->-=.5] (2,0)--(0,2);
\draw[black,thick,->-=.5] (0,2)--(0,.1);
\node[right] at (2,0) {$x_{1,0} \cdot a$};
\node[below, left] at (0,2) {$e_1 \cdot b$};
\node[above right] at (1,1) {$\sigma^{5}$};
\node[left] at (0,.75) {$\zeta^{5}$};
\draw[red] (0-.15,-.2)--(0-.1,2.1);
\draw[red] (.4-.15,-.2)--(.4-.1,2.1);
\draw[red] (.8-.15,-.2)--(.8-.1,2.1);
\draw[red] (1.2-.15,-.2)--(1.2-.1,2.1);
\draw[red] (1.6-.15,-.2)--(1.6-.1,2.1);
\draw[red] (2-.15,-.2)--(2-.1,2.1);
\end{tikzpicture}
\end{aligned}
\end{equation}
with the figure at the right illustrating the triangle that contributes the $i = 5$ term (for the sign, see \S\ref{sorp}).  This is just $b \cdot \theta_{1,0}[a]$ at $t = 1$ and $z = -\zeta$, it converges whenever $|\zeta|$ and $|a|$ are both less than one.

Thus the top row of \eqref{eq:only-square} is isomorphic to the map $\mathfrak{m} \to C$ sending $a$ to $\vartheta(a) := \sum_{n \in \bZ} (-\zeta)^n a^{p^{-n}}$.  This map obeys
\[
\vartheta(a^p) = (-\zeta) \vartheta(a)
\]
i.e. it intertwines the $\bF_p(\!(z)\!)$-module structure on $C$ given by the homomorphism $z \mapsto -\zeta$ with the $\bF_p(\!(z)\!)$-module structure on $\mathfrak{m}$ given by $(z,a) \mapsto a^p$.  The kernel is therefore an $\bF_p(\!(z)\!)$-module.  The image of this kernel under the isomorphism $\mathfrak{m} \cong B^{\varphi = z}$ (given by $a \mapsto \sum a^{p^i} z^{-i}$  \eqref{eq:feste}) is the set of $b \in B^{\varphi = z}$ whose set of zeroes is exactly $\{(-\zeta)^{p^n}\}_{n \in \bZ}$.  The function
\[
h(z) = \left(\sum_{i \in \bZ} a^{p^i} z^{-i}\right)\left(\prod_{n = 0}^{\infty} (1 + \zeta^{p^n}/z)\right)^{-1}
\]
(the meromorphic part of the Weierstrass factorization \cite[Ch. 2]{FF}) belongs to $C(\!(z)\!)$ and obeys the functional equation $h(z^{1/p})^p = (\zeta + z) h(z)$, i.e. its coefficients obey the recursion
\begin{equation}
\label{eq:hnzhn}
h_n^p - \zeta h_n = h_{n-1}
\end{equation}
---for each $h_{n-1}$ there are exactly $p$ solutions in $h_n$ to \eqref{eq:hnzhn}, so the set of such $h(z)$ is a one-dimensional $\bF_p(\!(z)\!)$-submodule of $C(\!(z)\!)$.

\subsection{Ore adjoint}
\label{subsec:ore}
Let $\Ls{\infty}^{\zeta}$ be as in \S\ref{subsec:sky}.  If we swap the order of $\Ls{\infty}$ and $\Ls{m}$ in \eqref{eq:pasteur}, the Maslov index of the intersection point is $1$, so that
\[
\CF(\Ls{\infty}^{\zeta},\Ls{m};\underline{C}) = \CF^1(\Ls{\infty}^{\zeta},\Ls{m};\underline{C}) =  e_m \cdot C
\]
The triangle sum
\[
\CF^1(\Ls{\infty}^{\zeta},\Ls{0};\underline{C}) \times \CF^0(\Ls{1},\Ls{\infty}^{\zeta};\underline{C}) \dashrightarrow \CF^1(\Ls{1},\Ls{0};\underline{C})
\]
is formally given by
\begin{equation}
\label{eq:hosono}
\sum_{n \in \bZ} (-1)^{3n} \sigma^{-n}(b \zeta^n a) = \sum_{n \in \bZ} (-\zeta)^{n p^n}{(ab)^{p^n}}
\end{equation}
It is the same triangles as \eqref{eq:melanogaster} that contribute to \eqref{eq:hosono}, but they are decorated differently.  For instance the triangle contributing the $n = 5$ summand is
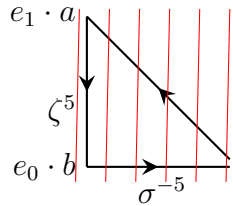
\begin{center}
\begin{tikzpicture}
\tikzset{->-/.style={decoration={ markings,
        mark=at position #1 with {\arrow[scale=1.5,>=stealth]{>}}},postaction={decorate}}}
\draw[black, thick, ->-=.5] (0,0)--(1.9,0);
\draw[black,thick,->-=.5] (1.9,.1)--(0,2);
\draw[black,thick,->-=.5] (0,2)--(0,0);
\node[above, left] at (0,2) {$e_1 \cdot a$};
\node[below, left] at (0,0) {$e_0 \cdot b$};
\node[below] at (1,0) {$\sigma^{-5}$};
\node[left] at (0,.75) {$\zeta^{5}$};
\draw[red] (0-.15,-.2)--(0-.1,2.1);
\draw[red] (.4-.15,-.2)--(.4-.1,2.1);
\draw[red] (.8-.15,-.2)--(.8-.1,2.1);
\draw[red] (1.2-.15,-.2)--(1.2-.1,2.1);
\draw[red] (1.6-.15,-.2)--(1.6-.1,2.1);
\draw[red] (2-.15,-.2)--(2-.1,2.1);
\end{tikzpicture}
\end{center}
Even if $|\zeta| < 1$, the $n \to -\infty$ tail of \eqref{eq:hosono} does not converge unless $ab = 0$.  Even so, it is interesting in a formal way.  In \cite{Poonen},  Poonen following \cite{Ore} attaches to each series of the form $f(a) = \sum u_n a^{p^n}$ an ``adjoint'' series $f^{\dagger}(a) := \sum u_{-n}^{p^n} a^{p^n}$ --- let us call it the Ore adjoint.  Evidently \eqref{eq:hosono} is exactly $\vartheta^\dagger[ba]$.

Under some hypotheses on $f$ Poonen shows that the kernels of $f$ and $f^\dagger$ are Pontrjagin dual to each other in a canonical fashion.  These hypotheses are not satisfied by $\vartheta(a)$, but as $\ker(\vartheta)$ (being the additive group of a local field) is Pontrjagin self-dual, and as $\vartheta^\dagger$ does not converge in any case, we are perhaps free to speculate that ``$\ker(\vartheta^\dagger)$'' is somehow morally isomorphic to $\bF_p(\!(z)\!)$.  This speculation is consistent with mirror symmetry: on the Fargues-Fontaine curve there indeed is a short exact sequence
\begin{equation}
\label{eq:ore-ses}
0 \to \Hom(\cO(1),\cO(1)) \to \Hom(\cO(1),\delta) \to \Ext^1(\cO(1),\cO) \to 0
\end{equation}
coming from the resolution $\cO \to \cO(1)$ of the skyscraper sheaf $\delta$, and the vanishing of $\Ext^1(\cO(1),\cO(1)) = H^1(\FF;\cO)$.  The kernel of \eqref{eq:ore-ses} is naturally isomorphic to $H^0(\FF;\cO)$, i.e. to $\bF_p(\!(z)\!)$.  The middle group is isomorphic to $C$ and to $\HF^0(\Ls{1},\Ls{\infty}^\zeta;\underline{C})$.  But we emphasize that $\Ext^1(\cO(1),\cO)$ is not isomorphic to $\HF(\Ls{1},\Ls{0};\underline{C})$, nor to any open subgroup of it.   

\subsection{Loud Floer cochains on $\Ls{0}$}
\label{subsec:49}
Let $\{\phi^s\}_{s \in \bR}$ be as in \eqref{eq:phis}:
\begin{equation}
\phi^s(x,y) = (x,y - s \cos(2\pi x))
\end{equation}
We can try to compare $\CF(L,L';\underline{C})$ and $\CF(\phi^s L,L';\underline{C})$ by specializing to $t = 1$ in \eqref{eq:cont-twist}.  In some cases, for instance if $L$ and $L'$ are parallel to $\Ls{0}$ as in \S\ref{subsec:ubiquinone}, the summation \eqref{eq:cont-twist} is finite and defines a map
\[
\CF(L,L';\underline{C}) \to \CF(\phi^s L,L';\underline{C})
\]
without any problems.  But this map is not always a quasi-isomorphism.  The series defining the homotopy \eqref{eq:delta-twist} may not converge at $t = 1$ --- \eqref{eq:dhap} is a vivid example of this.  

We will analyze the continuation maps between the groups $\CF(\phi^s \Ls{0},\Ls{0};\underline{C})$.  If $n$ is an integer and $n < s < n+1$, then $\Ls{0}$ meets $\phi^s \Ls{0}$ in $2n+2$ points.  In the fundamental domain $[0,1] \times [0,1]$, half of them have $x$-coordinate $<0.5$ and half of them have $x$-coordinate $>0.5$.  They are linearly ordered by the $x$-coordinate and after listing them in that order we will name them
\[
z_{n}^{(s)},\ldots,z_{-n}^{(s)},\xi_{-n}^{(s)},\ldots,\xi_{n}^{(s)}
\]
More explicitly, $z_i^{(s)}$ and $\xi_i^{(s)}$ are the two solutions to 
$
i - s \cos(2 \pi x) = 0
$, i.e. for a suitable branch of the inverse cosine function:
\[
z_i^{(s)} = \frac{1}{2\pi} \arccos(i/s) \qquad \xi_i^{(s)} = 1- \frac{1}{2\pi} \arccos(i/s)
\]
The case $1 < s < 2$ is shown in the diagram, along with orientations and stars:
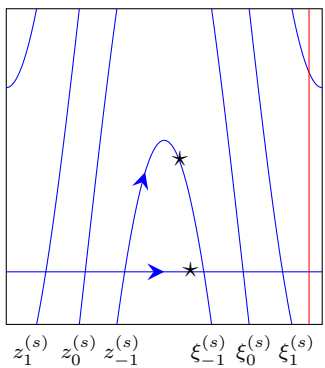
\begin{center}
\begin{tikzpicture}[scale=0.7]
      \tikzset{->-/.style={decoration={ markings,
        mark=at position #1 with {\arrow[scale=2,>=stealth]{>}}},postaction={decorate}}}
         \draw (0,-1)--(6,-1)--(6,6-1)--(0,6-1)--(0,-1);
         \draw[blue, ->-=.5]  (0,0) -- (6,0);
         \draw[domain=0:34.56,smooth,variable=\x,blue] plot ({\x/60},{-8.5*cos(\x)+12});
         \draw[domain=34.56:83.24,smooth,variable=\x,blue] plot ({\x/60},{-8.5*cos(\x)+6});
         \draw[domain=83.24:126.02,smooth,variable=\x,blue] plot ({\x/60},{-8.5*cos(\x)});
         \draw[domain=126.02:234,smooth,variable=\x,blue, ->-=.4] plot ({\x/60},{-8.5*cos(\x)-6});
          \draw[domain=234:276.76,smooth,variable=\x,blue] plot ({\x/60},{-8.5*cos(\x)});
          \draw[domain=276.76:325.44,smooth,variable=\x,blue] plot ({\x/60},{-8.5*cos(\x)+6});
          \draw[domain=325.44:360,smooth,variable=\x,blue] plot ({\x/60},{-8.5*cos(\x)+12}); 
         \node at (3.5,0.05) {$\star$};
         \node at (3.3,2.15) {$\star$};
         \draw[red] (5.75,-1)--(5.75,5);
         \node[below] at (34.56/60-.1,-1) {\tiny $z_{1}^{(s)}$};
         \node[below] at (83.24/60,-1) {\tiny $z_0^{(s)}$};
         \node[below] at (126/60+.1,-1) {\tiny $z_{-1}^{(s)}$};
         \node[below] at (6-34.56/60+.1,-1) {\tiny $\xi_{1}^{(s)}$};
         \node[below] at (6-83.24/60+.1,-1) {\tiny $\xi_0^{(s)}$};
         \node[below] at (6-126/60-.05,-1) {\tiny $\xi_{-1}^{(s)}$};
\end{tikzpicture}
\end{center}
The rules of \S\ref{subsec:mas} give each point $z_i^{(s)}$ the Maslov index $0$ and each $\xi_i^{(s)}$ the Maslov index $1$, so that 
\[
\CF^0(\phi^s \Ls{0},\Ls{0};\underline{C}) = \bigoplus_{i = -n}^n z_i^{(s)}\cdot C \qquad \CF^1(\phi^s \Ls{0},\Ls{0};\underline{C}) = \bigoplus_{i = -n}^n \xi_i^{(s)} \cdot C 
\]
Each $\xi_i^{(s)}$ is the output vertex of exactly two bigons, and the other vertex of both bigons in $z_i^{(s)}$.  There is an ``upward'' bigon whose boundary passes through $(0.5,s)+\bZ^2$ and a ``downward'' on whose boundary passes through $(0,-s)+\bZ^2$.  If one places a star at (or close to, as in the figure above) $(0.5,s)+\bZ^2$ and $(0,-s)+\bZ^2$, then the sign of every downward bigon is $1$ and the sign of every upward bigon is $-1$.  The downward bigons cross the danger line exactly once, and the upward bigons exactly never, so that
\[
\mu_1(z_i^{(s)} \cdot a) = \xi_i^{(s)} \cdot (-a + \sigma(a) ) 
\]
with the upward bigon contributing $-a$ and the downward bigon contributing $\sigma(a)$.

If $s' > s$ then for a suitable choice of profile function $\beta$ (on that is very close to the ``linear cascades'' limit considered in \cite{Auroux}), the continuation map
\begin{equation}
\label{eq:landsteiner}
\CF(\phi^s \Ls{0},\Ls{0};\underline{C}) \to \CF(\phi^{s'} \Ls{0},\Ls{0};\underline{C})
\end{equation}
simply sends $z_i^{(s)} \cdot a$ to $z_i^{(s')} \cdot a$ and $\xi_i^{(s)} \cdot a$ to $\xi_i^{(s')} \cdot a$.  In particular it defines a filtered diagram of cochain complexes (indexed by $s > 0$, $s \notin \bZ$, with respect to the usual ordering of real numbers $s$.  Let $\CF_{\mathrm{loud}}(\Ls{0},\Ls{0})$ denote the direct limit of this diagram 
\[
\CF_{\mathrm{loud}}(\Ls{0},\Ls{0};\underline{C}) := \varinjlim_{s > 0| s \notin \bZ}  \CF(\phi^s \Ls{0},\Ls{0};\underline{C})
\]
Since each map \eqref{eq:landsteiner} is the inclusion of a direct summand of cochain complexes, $\CF_{\mathrm{loud}}$ is a model for the homotopy colimit of cochain complexes as well.  Explicitly,
\begin{equation}
\label{eq:cfloud-mu1}
\CF_{\mathrm{loud}}^0 = \bigoplus_{i \in \bZ} z_i \cdot  C \qquad \CF_{\mathrm{loud}}^1 = \bigoplus_{i \in \bZ}  \xi_{i} \cdot C \qquad \mu_1(z_i \cdot a) = \xi_i \cdot (-a+\sigma(a))
\end{equation}

\subsection{Triangles between the $\phi^s \Ls{0}$}
\label{subsec:penultimate}
Continuing with the notation of \S\ref{subsec:49}, let us suppose that none of $s$, $s'$, and $s+s'$ are in $\bZ$, and describe the triangles between $\Ls{0},\phi^{s} \Ls{0}$, and $\phi^{s+s'}\Ls{0}$.  The ``output'' corners of these triangles are
\[z_i^{(s+s')}
\text{ and }\xi_i^{(s+s')}\text{ on }\Ls{0} \cap \phi^{s+s'} \Ls{0},\] and the other two corners in counterclockwise order are
\[
\phi^{s} z_i^{(s')}, \phi^{s} \xi_i^{(s')} \in \phi^{s} \Ls{0} \cap \phi^{s+s'} \Ls{0}, \qquad z_i^{(s)},\xi_i^{(s)} \in \Ls{0} \cap \phi^s \Ls{0}.
\]
For each such triangle there is a unique pair of integers $i$ and $j$ so that the triangle lifts to $\bR^2$ with boundary on the $x$-axis, the graph of $y = i - s \cos(2 \pi x)$, and the graph of $y = i+j - (s+s') \cos(2\pi x)$.  The only non-empty moduli spaces of triangles are
\begin{equation}
\label{eq:monomial-moduli}
\cM\left(z_{i+j}^{(s+s')},\phi^{s} z_j^{(s')},z_i^{(s)}\right) \quad \cM\left(\xi_{i+j}^{(s+s')},\phi^{s} z_j^{(s')},\xi_i^{(s)}\right) \quad \cM\left(\xi_{i+j}^{(s+s')},\phi^{s} \xi_j^{(s')},z_i^{(s)}\right) 
\end{equation}
with $-s < i < s$ and $-s' < j < s'$.  When $i/s = j/s'$, there is something tricky about the latter two moduli spaces (they are not transversely cut \S\ref{tcc}), so let us for a moment assume that $i/s \neq j/s'$.  Then each space \eqref{eq:monomial-moduli} contains exactly one triangle.  The nature of this triangle depends on which of $i/s$ or $j/s'$ is larger.
\medskip

The triangle of $\cM\left(z_{i+j}^{(s+s')},\phi^{s} z_j^{(s')},z_i^{(s)}\right)$ is the one bounded by
\begin{eqnarray}
\max(0,(i+j) - (s+s')\cos(2\pi x)) \leq y \leq i - s \cos(2 \pi x) & \text{if $j/s' < i/s$} \label{eq:4103} \\
\min(0,(i+j) - (s+s') \cos(2\pi x)) \geq y \geq i - s\cos(2 \pi x) & \text{if $j/s' > i/s$} \label{eq:4104}
\end{eqnarray} 
For legibility, in the following illustration of these triangles the curves $\phi^s \Ls{0}$ and $\phi^{s+s'} \Ls{0}$ are drawn in a different aspect ratio than in the diagram of \S\ref{subsec:49}, vertically compressed.  The figure is drawn in a union of $\sim s+s'$ fundamental domains, stacked on top of each other. 
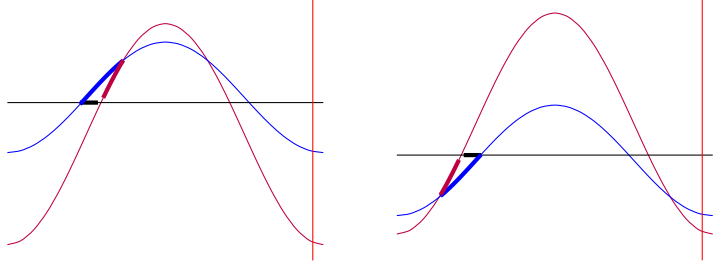
\begin{center}
\begin{tikzpicture}[scale = 0.7]
  \draw[domain=0:6,smooth,variable=\x,purple] plot ({\x},{-.6-2.1*cos(60*\x)});
  \draw[domain=0:6,smooth,variable =\x,blue] plot ({\x},{.1-1.05*cos(60*\x)});
  \draw[domain=0:6,smooth,variable = \x,black] plot ({\x},0);
  \draw[domain=1.38:1.8-.08,smooth,variable = \x,black,ultra thick] plot ({\x},0); %.08
  \draw[domain=1.38:2.2,smooth,variable =\x,blue,ultra thick] plot ({\x},{.1-1.05*cos(60*\x)});
  \draw[domain=1.76+.06:2.2,smooth,variable=\x,purple,ultra thick] plot ({\x},{-.6-2.1*cos(60*\x)});
  \draw[red] (5.8,-3)--(5.8,2);  %.06
\end{tikzpicture}
\qquad
\begin{tikzpicture}[scale = 0.7]
  \draw[domain=0:6,smooth,variable=\x,purple] plot ({\x},{.6-2.1*cos(60*\x)});
  \draw[domain=0:6,smooth,variable =\x,blue] plot ({\x},{-.1-1.05*cos(60*\x)});
  \draw[domain=0:6,smooth,variable = \x,black] plot ({\x},0);
  \draw[domain=1.19+.08:1.6,smooth,variable = \x,black,ultra thick] plot ({\x},0); % .08
  \draw[domain=.83:1.6,smooth,variable =\x,blue,ultra thick] plot ({\x},{-.1-1.05*cos(60*\x)});
  \draw[domain=.83:1.22-.04,smooth,variable=\x,purple,ultra thick] plot ({\x},{.6-2.1*cos(60*\x)});
  \draw[red] (5.8,-2)--(5.8,3); %.04
\end{tikzpicture}
\end{center}
In this and the following diagrams, $\phi^{s+s'}\Ls{0}$ is purple, $\phi^s \Ls{0}$ is blue, and $\Ls{0}$ is black.  The left side shows the typical case where $i/s > j/s'$, and the right side shows the typical case when $i/s < j/s'$.

In $\cM\left(\xi_{i+j}^{(s+s')},\phi^{s} z_j^{(s')},\xi_i^{(s)}\right)$ we have the triangle 
\begin{equation}
\label{eq:4105}
\begin{array}{cl}
 (\ref{eq:4103}) \cup \{ \min(0,i - s\cos(2 \pi x)) \geq y \geq (i+j) - (s+s') \cos(2 \pi x) \} & \text{if } \frac{j}{s'}<\frac{i}{s} \\
 (\ref{eq:4104}) \cup \{ \max(0,i -  s \cos(2 \pi x)) \leq y \leq (i+j) - (s+s') \cos(2 \pi x) \} & \text{if }\frac{j}{s'} >\frac{i}{s}
\end{array}
\end{equation}
\begin{center}
\begin{tikzpicture}[scale = 0.7]
  \draw[domain=-3:3,smooth,variable=\x,purple] plot ({\x},{-.6-2.1*cos(60*\x)});
  \draw[domain=-3:3,smooth,variable =\x,blue] plot ({\x},{.1-1.05*cos(60*\x)});
  \draw[domain=-3:3,smooth,variable = \x,black] plot ({\x},0);
  \draw[domain=-1.38:-1.8+.09,smooth,variable = \x,black,ultra thick] plot ({\x},0);%.09
  \draw[domain=-1.38:2.2,smooth,variable =\x,blue,ultra thick] plot ({\x},{.1-1.05*cos(60*\x)});
  \draw[domain=-1.76+.06:2.2,smooth,variable=\x,purple,ultra thick] plot ({\x},{-.6-2.1*cos(60*\x)});%.06
  \draw[red] (5.8-6,-3)--(5.8-6,2);
\end{tikzpicture}
\qquad
\begin{tikzpicture}[scale = 0.7]
  \draw[domain=0:6,smooth,variable=\x,purple] plot ({\x},{.6-2.1*cos(60*\x)});
  \draw[domain=0:6,smooth,variable =\x,blue] plot ({\x},{-.1-1.05*cos(60*\x)});
  \draw[domain=0:6,smooth,variable = \x,black] plot ({\x},0);
  \draw[domain=6-1.6:6-1.19-.09,smooth,variable = \x,black,ultra thick] plot ({\x},0);%.09
  \draw[domain=.83:6-1.6,smooth,variable =\x,blue,ultra thick] plot ({\x},{-.1-1.05*cos(60*\x)});
  \draw[domain=.83:6-1.22-.06,smooth,variable=\x,purple,ultra thick] plot ({\x},{.6-2.1*cos(60*\x)});%.06
  \draw[red] (5.8,-2)--(5.8,3);
\end{tikzpicture}
\end{center}

In $\cM\left(\xi_{i+j}^{(s+s')},\phi^{s} \xi_j^{(s')},z_i^{(s)}\right)$ we have the triangle
\begin{equation}
\label{eq:4106}
\begin{array}{cl}
(\ref{eq:4103}) \cup \{ \min(i -  s \cos(2 \pi x), i+j- (s+s') \cos(2\pi x))) \geq y \geq 0 \} & \text{if }\frac{j}{s'} <\frac{i}{s} \\ 
(\ref{eq:4104}) \cup \{ \max(i - s\cos(2 \pi x), i+j - (s+s') \cos(2\pi x)) \leq y \leq 0 \} & \text{if } \frac{j}{s'}>\frac{i}{s} 
\end{array}
\end{equation}

\begin{center}
\begin{tikzpicture}[scale = 0.7]
  \draw[domain=0:6,smooth,variable=\x,purple] plot ({\x},{-.6-2.1*cos(60*\x)});
  \draw[domain=0:6,smooth,variable =\x,blue] plot ({\x},{.1-1.05*cos(60*\x)});
  \draw[domain=0:6,smooth,variable = \x,black] plot ({\x},0);
  \draw[domain=1.38:6-1.8-.08,smooth,variable = \x,black,ultra thick] plot ({\x},0);%.08
  \draw[domain=1.38:6-2.2,smooth,variable =\x,blue,ultra thick] plot ({\x},{.1-1.05*cos(60*\x)});
  \draw[domain=6-2.2:6-1.76-.06,smooth,variable=\x,purple,ultra thick] plot ({\x},{-.6-2.1*cos(60*\x)});%.06
  \draw[red] (5.8,-3)--(5.8,2);
\end{tikzpicture}
\qquad
\begin{tikzpicture}[scale = 0.7]
  \draw[domain=0-3:6-3,smooth,variable=\x,purple] plot ({\x},{.6-2.1*cos(60*\x)});
  \draw[domain=0-3:6-3,smooth,variable =\x,blue] plot ({\x},{-.1-1.05*cos(60*\x)});
  \draw[domain=0-3:6-3,smooth,variable = \x,black] plot ({\x},0);
  \draw[domain=-1.19+.08:1.6,smooth,variable = \x,black,ultra thick] plot ({\x},0);%.08
  \draw[domain=-.83:1.6,smooth,variable =\x,blue,ultra thick] plot ({\x},{-.1-1.05*cos(60*\x)});
  \draw[domain=-1.22+.06:-.83,smooth,variable=\x,purple,ultra thick] plot ({\x},{.6-2.1*cos(60*\x)});
  \draw[red] (5.8-6,-2)--(5.8-6,3);%.06
\end{tikzpicture}
\end{center}

\medskip

Now we discuss the triangles with $i/s = j/s'$.  For generic $s$ and $s'$, it is only possible that $i/s = j/s'$ when $i = j = 0$.  In that case $\cM\left(z_0^{(s+s')},\phi^s z_0^{(s')}, z_0^{(s)}\right)$ again contains a single point (the constant map with value $z_0 = (.25,0)$), and is again transversely cut, but these two assertions are not true for $\cM\left(\xi_{0}^{(s+s')},\phi^{s} z_0^{(s')},\xi_0^{(s)}\right)$ or for $\cM\left(\xi_{0}^{(s+s')},\phi^{s} \xi_0^{(s')},z_0^{(s)}\right)$.  These spaces each contain two points, which are degenerate triangles (they are at the boundary of the Deligne-Mumford-Stasheff compactification) which are not maps out of a triangle but out of a wedge sum of triangle and a bigon:
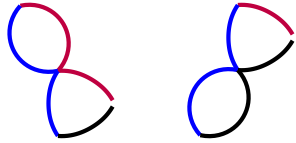
\begin{center}
\begin{tikzpicture}[scale = 0.5]
   \draw [purple, ultra thick, domain = 5:120] plot ({(1-(\x/120)*(1-\x/120)) * cos(\x)}, {(1-(\x/120)*(1-\x/120)) *sin(\x)});
   \draw [blue, ultra thick, domain = 0:120] plot ({(1-(\x/120)*(1-\x/120)) * cos(\x+120)}, {(1-(\x/120)*(1-\x/120)) *sin(\x+120)});
    \draw [black, ultra thick, domain = 0:115] plot ({(1-(\x/120)*(1-\x/120)) * cos(\x+240)}, {(1-(\x/120)*(1-\x/120)) *sin(\x+240)});
    \draw [purple, ultra thick, domain = 0:180] plot ({2*cos(120)+(1-(\x/180)*(1-\x/180))*cos(\x-60)},{2*sin(120)+(1-(\x/180)*(1-\x/180))*sin(\x-60)});
    \draw [blue, ultra thick, domain = 0:180] plot ({2*cos(120)+(1-(\x/180)*(1-\x/180))*cos(\x+180-60)},{2*sin(120)+(1-(\x/180)*(1-\x/180))*sin(\x+180-60)}); 
\end{tikzpicture}
\qquad
\begin{tikzpicture}[scale = 0.5]
   \draw [purple, ultra thick, domain = 5:120] plot ({(1-(\x/120)*(1-\x/120)) * cos(\x)}, {(1-(\x/120)*(1-\x/120)) *sin(\x)});
   \draw [blue, ultra thick, domain = 0:120] plot ({(1-(\x/120)*(1-\x/120)) * cos(\x+120)}, {(1-(\x/120)*(1-\x/120)) *sin(\x+120)});
    \draw [black, ultra thick, domain = 0:115] plot ({(1-(\x/120)*(1-\x/120)) * cos(\x+240)}, {(1-(\x/120)*(1-\x/120)) *sin(\x+240)});
    \draw [blue, ultra thick, domain = 0:180] plot ({2*cos(120)+(1-(\x/180)*(1-\x/180))*cos(\x+60)},{-2*sin(120)+(1-(\x/180)*(1-\x/180))*sin(\x+60)});
    \draw [black, ultra thick, domain = 0:180] plot ({2*cos(120)+(1-(\x/180)*(1-\x/180))*cos(\x+180+60)},{-2*sin(120)+(1-(\x/180)*(1-\x/180))*sin(\x+180+60)}); 
\end{tikzpicture}
\end{center}
The degenerate maps in $\cM\left(\xi_{0}^{(s+s')},\phi^{s} z_0^{(s')},\xi_0^{(s)}\right)$ and $\cM\left(\xi_{0}^{(s+s')},\phi^{s} \xi_0^{(s')},z_0^{(s)}\right)$ collapse the triangle part to a point (to $\xi_0 = (.75,0)$) but are nontrivial along the bigon:

\begin{equation}
\label{eq:two-by-two}
\begin{aligned}
\begin{tikzpicture}[scale = 0.7]
  \draw[domain=-3:3,smooth,variable=\x,purple] plot ({\x},{-2.1*cos(60*\x)});
  \draw[domain=-3:3,smooth,variable =\x,blue] plot ({\x},{-1.05*cos(60*\x)});
  \draw[domain=-3:3,smooth,variable = \x,black] plot ({\x},0);
  \draw[domain=-1.5:1.5,smooth,variable =\x,blue,ultra thick] plot ({\x},{-1.05*cos(60*\x)});
  \draw[domain=-1.5:1.5,smooth,variable=\x,purple,ultra thick] plot ({\x},{-2.1*cos(60*\x)});%.06
  \node at (-1.5,0) {$\blacktriangle$};
  \draw[red] (5.8-6,-2.5)--(5.8-6,2.5);
\end{tikzpicture}
& \qquad &
\begin{tikzpicture}[scale = 0.7]
  \draw[domain=0:6,smooth,variable=\x,purple] plot ({\x},{-2.1*cos(60*\x)});
  \draw[domain=0:6,smooth,variable =\x,blue] plot ({\x},{-1.05*cos(60*\x)});
  \draw[domain=0:6,smooth,variable = \x,black] plot ({\x},0);
  \draw[domain=1.5:6-1.5,smooth,variable =\x,blue,ultra thick] plot ({\x},{-1.05*cos(60*\x)});
  \draw[domain=1.5:6-1.5,smooth,variable=\x,purple,ultra thick] plot ({\x},{-2.1*cos(60*\x)});
  \node at (4.5,0) {$\blacktriangle$};
  \draw[red] (5.8,-2.5)--(5.8,2.5);
\end{tikzpicture}
\\
\\
\begin{tikzpicture}[scale = 0.7]
  \draw[domain=0:6,smooth,variable=\x,purple] plot ({\x},{-2.1*cos(60*\x)});
  \draw[domain=0:6,smooth,variable =\x,blue] plot ({\x},{-1.05*cos(60*\x)});
  \draw[domain=0:6,smooth,variable = \x,black] plot ({\x},0);
  \draw[domain=1.5:6-1.5,smooth,variable = \x,black,ultra thick] plot ({\x},0);%.08
  \draw[domain=1.5:6-1.5,smooth,variable =\x,blue,ultra thick] plot ({\x},{-1.05*cos(60*\x)});
  \node at (4.5,0) {$\blacktriangle$};
  \draw[red] (5.8,-2.5)--(5.8,2.5);
\end{tikzpicture}
& \qquad &
\begin{tikzpicture}[scale = 0.7]
  \draw[domain=0-3:6-3,smooth,variable=\x,purple] plot ({\x},{-2.1*cos(60*\x)});
  \draw[domain=0-3:6-3,smooth,variable =\x,blue] plot ({\x},{-1.05*cos(60*\x)});
  \draw[domain=0-3:6-3,smooth,variable = \x,black] plot ({\x},0);
  \draw[domain=-1.5:1.5,smooth,variable = \x,black,ultra thick] plot ({\x},0);%.08
  \draw[domain=-1.5:1.5,smooth,variable =\x,blue,ultra thick] plot ({\x},{-1.05*cos(60*\x)});
  \draw[red] (5.8-6,-2.5)--(5.8-6,2.5);%.06
  \node at (-1.5,0) {$\blacktriangle$};
\end{tikzpicture}
\end{aligned}
\end{equation}
The top two figures indicate the two points of  $\cM\left(\xi_{0}^{(s+s')},\phi^{s} z_0^{(s')},\xi_0^{(s)}\right)$, the bottom two are the two points of $\cM\left(\xi_{0}^{(s+s')},\phi^{s} \xi_0^{(s')},z_0^{(s)}\right)$.  Though they are not transversely cut they have analytic index zero --- more precisely they have index $+1$ along the constant triangle and index $-1$ along the bigon.

\subsection{Triangle products on $\CF_{\mathrm{loud}}$}
For short, let us put $A^{(s)} := \CF(\phi^s \Ls{0},\Ls{0};\underline{C})$.  The triangles in the previous section, together with the identification  $\CF(\phi^{s+s'},\phi^{s} \Ls{0};\underline{C})$ of and $\CF(\phi^{s'} \Ls{0},\Ls{0};\underline{C})$, give a multiplication
$A^{(s)} \times A^{(s')} \to A^{(s+s')}$
specifically 
\begin{itemize}
\item (Coming from \eqref{eq:4103} and \eqref{eq:4104})
\[
(z_i^{(s)} \cdot a, z_j^{(s')} \cdot b) \mapsto z_{i+j}^{(s+s')} \cdot ab\]
\item (Coming from \eqref{eq:4105})
\begin{equation}
\label{eq:4105p}
(\xi_i^{(s)} \cdot a,z_j^{(s')} \cdot b) \mapsto \xi_{i+j}^{(s+s')}\cdot \begin{cases}
a\sigma(b)  & \text{if $j/s' < i/s$} \\
ab & \text{if $j/s' > i/s$}
\end{cases}
\end{equation}
\item (Coming from \eqref{eq:4106})
\begin{equation}
\label{eq:4106p}
(z_i^{(s)} \cdot a, \xi_j^{(s')} \cdot b) \mapsto \xi_{i+j}^{(s+s')} \cdot \begin{cases}
ab & \text{if $j/s' < i/s$} \\
\sigma(a)b & \text{if $j/s' > i/s$}
\end{cases}
\end{equation}
\end{itemize}

Since  $\cM\left(\xi_{0}^{(s+s')},\phi^{s} z_0^{(s')},\xi_0^{(s)}\right)$ and $\cM\left(\xi_{0}^{(s+s')},\phi^{s} \xi_0^{(s')},z_0^{(s)}\right)$ are not transversely cut, they carry a virtual fundamental class rather than an orientation.  We will simply put
\begin{equation}
\label{eq:00}
(\xi_0^{(s)} \cdot a, z_0{(s')} \cdot b) \mapsto \xi_0^{(s+s')} \cdot ab \qquad (z_0^{(s)} a, \xi_0^{(s')} b) \mapsto \xi_0^{(s+s')} \cdot \sigma(a)b
\end{equation}
as though the left two degenerate triangles displayed in \eqref{eq:two-by-two} contributed nothing. The same issue can also be addressed by introducing a Hamiltonian perturbation $\psi$ of $\Ls{0}$ (but not $\phi^s \Ls{0}$ or $\phi^{s+s'} \Ls{0})$ supported in a very small neighborhood of $z_0$ and $\xi_0$.  In that case all moduli spaces are transversely cut and the triangle products $\mu_2(\psi z_0^{(s)} \cdot a, \phi^s \xi_0^{(s')} \cdot b)$ and $\mu_2(\psi \xi_0^{(s)} \cdot a,\phi^s z_0^{(s')} \cdot b)$ are well-defined, though the specific formula will depend on $\psi$ --- \eqref{eq:00} is consistent with some of these $\psi$.
\medskip

Now we use the products $A^{(s)} \times A^{(s')} \to A^{(s+s')}$ to define a multiplication on $\lim_s A^{(s)}$, i.e. on 
$
\CF_{\mathrm{loud}}(\Ls{0},\Ls{0};\underline{C})$.  It is not quite straightforward, because the products \eqref{eq:4105p} and \eqref{eq:4106p} are not eventually constant as $s$ and $s'$ grow --- it depends on which of $i/s$ and $j/s'$ are larger.  The square
\begin{equation}
\label{eq:asas}
\begin{aligned}
\xymatrix{
A^{(s)} \times A^{(s')} \ar[d] \ar[r] & A^{(s+s')} \ar[d] \\
A^{(S)} \times A^{(S')} \ar[r] & A^{(S+S')}
}
\end{aligned}
\end{equation}
does not commute for all $s,s',S,S'$ with $s < S$ and $s' < S'$.  We address this in the following crude way: we choose an irrational number $e >0$, and note that since $i/s - j/(es)$ has constant sign for $s > 0$,  \eqref{eq:asas} does commute when $s' = es$ and $S' = eS$.  The induced multiplication on the colimit is explicitly $(z_i \cdot a, z_j \cdot b) \mapsto z_{i+j} \cdot ab$ and 
\[
(\xi_i \cdot a,z_j \cdot b) \mapsto \xi_{i+j} \cdot \begin{cases}
a \sigma(b) & \text{if $j/e < i$} \\
ab & \text{if $j/e \geq i$} \\
\end{cases}
\qquad (z_i \cdot a, \xi_j \cdot b) \mapsto \xi_{i+j} \cdot \begin{cases}
ab & \text{if $j/e < i$} \\
\sigma(a)b & \text{if $j/e \geq i$}
\end{cases}
\]
Thus we get one binary operation on $\CF_{\mathrm{loud}}(\Ls{0},\Ls{0};\underline{C})$ for every irrational $e > 0$.  These multiplications are genuinely different for different $e$.  Moreover, they are not associative; they do however, obey the Leibniz rule $\mu_1(w w') = \mu_1(w) w' + w \mu_1(w')$, with $\mu_1$ as in \eqref{eq:cfloud-mu1}.  It is likely that they can be extended to an $A_{\infty}$-structure on $\CF_{\mathrm{loud}}(\Ls{0},\Ls{0};\underline{C})$ (and even more likely that there is such an $A_{\infty}$-structure on a complex quasi-isomorphic to it, defined along the lines of \cite{Abouzaid-Seidel}), but we will not construct it.  Instead we simply note that the induced multiplication on $\HF^0_{\mathrm{loud}} := \ker(\mu_1)$ (and even on $\HF^0_{\mathrm{loud}} \oplus \HF^1_{\mathrm{loud}}$, though the degree $1$ part vanishes if $C$ is algebraically closed and $\sigma$ is the $p$th root map), is associative, and independent of $e$.  Indeed it is simply the Laurent polynomial ring $C^{\sigma}[z^{\pm 1}]$ under the assignment $\sum c_i z^i \mapsto \sum z_i \cdot c_i$.

\end{document}